\documentclass{article}
\usepackage{amssymb,amsthm, amsmath}
\usepackage{graphicx}
\usepackage{color}

\usepackage[a4paper,bindingoffset=0.2in,%
            left=1in,right=1in,%
            footskip=.25in]{geometry}

\begin{document}

\newtheorem{lemma}{Lemma}
\newtheorem{theorem}{Theorem}
\newtheorem{cor}{Corollary}
\newtheorem{defn}{Definition}
\newtheorem{remark}{Remark}

\newcommand{\nwc}{\newcommand}
\nwc{\Levy}{L\'{e}vy}
\nwc{\Holder}{H\"{o}lder}
\nwc{\cadlag}{c\`{a}dl\`{a}g}
\nwc{\be}{\begin{equation}}
\nwc{\ee}{\end{equation}}
\nwc{\ba}{\begin{eqnarray}}
\nwc{\ea}{\end{eqnarray}}
\nwc{\la}{\label}
\nwc{\nn}{\nonumber}
\nwc{\Z}{\mathbb{Z}}
\nwc{\C}{\mathbb{C}}
\nwc{\E}{\mathbb{E}}
\nwc{\R}{\mathbb{R}}
\nwc{\N}{\mathbb{N}}
\nwc{\prob}{\mathbb{P}}
\nwc{\Skor}{\mathbb{D}}
\nwc{\PP}{\mathcal{P}}
\nwc{\M}{\mathcal{M}}
\nwc{\law}{\stackrel{\mathcal{L}}{\rightarrow}}
\nwc{\eqd}{\stackrel{\mathcal{L}}{=}}
\nwc{\vp}{\varphi}
\nwc{\Vp}{\Phi}
\nwc{\veps}{\varepsilon}
\nwc{\eps}{\ve}
\nwc{\qref}[1]{(\ref{#1})}
\nwc{\D}{\partial}
\nwc{\dnto}{\downarrow}
\nwc{\nsup}{^{(n)}}
\nwc{\ksup}{^{(k)}}
\nwc{\jsup}{^{(j)}}
\nwc{\nksup}{^{(n_k)}}
\nwc{\inv}{^{-1}}
\nwc{\one}{\mathbf{1}}
\nwc{\argmin}{\mathrm{arg}^+\mathrm{min}}
\nwc{\argmax}{\mathrm{arg}^+\mathrm{max}}
\nwc{\Rplus}{\R_+}
\nwc{\Rorder}{\R_<}
\nwc{\xx}{\mathbf{x}}
\nwc{\emp}{\mu}
\nwc{\empN}{F^n}
\nwc{\lossN}{L^n}
\nwc{\Lip}{\mathrm{Lip}}
\nwc{\BL}{\mathrm{BL}}
\nwc{\ddm}{d}
\nwc{\dif}{D}

\title{Concentration inequalities  for the  hydrodynamic limit of a two-species stochastic particle system  }

\author{Joseph Klobusicky\footnote{Department of Mathematics, The University of Scranton, Scranton PA, 18510  (\texttt{joseph.klobusicky@scranton.edu}).}}
\date{}
\maketitle

\newcommand{\no}{\textcolor{red}{[CLUNK!] }} 

\begin{abstract}
We study a stochastic particle system which is motivated from
grain boundary coarsening in two-dimensional networks. Each particles lives on the positive real line and is labeled as belonging to either Species 1 or Species 2. Species 1 particles drift at unit speed toward the origin, while Species 2 particles do not move. When a particle in Species 1 hits the origin, it is removed, and a randomly selected particle mutates from  Species 2 to Species 1. The process described is an example of a high-dimensional piecewise deterministic Markov process (PDMP), in which deterministic flow is punctuated with stochastic jumps.  Our main result is a proof of exponential concentration inequalities of the Kolmogorov-Smirnoff distance between empirical measures of the particle system and solutions of limiting  nonlinear kinetic equations. Our method of proof involves a time and space discretization of the kinetic equations, which we compare with the particle system to derive  recurrence inequalities for comparing total numbers in small intervals. To show these recurrences occur with high probability, we  appeal to a state dependent Hoeffding type inequality at each time increment. \end{abstract}


\section{Introduction}

An important topic in  material science is    the coarsening
of network microstructures such as polycrystalline metals and soap froths.
 Through  heating of metals or gas diffusion of foams, coarsening
 is induced from the migration of network boundaries to minimize interfacial
surface energy.
While    tracking individual boundaries remains a multifaceted and active field
of research in numerical \cite{esedouglu2018algorithms} and geometric   \cite{ilmanen2019short} analysis,   in the 1950's von Neumann
\cite{von1952discussion}
 and Mullins \cite{mul56} proved a simple relation between the topology and
geometry
of a single cell in two-dimensional networks with isotropic surface tension evolving through curve-shortening flow.
  Specifically, a cell with area $A$ and $n$ sides has  a constant growth
rate 
\begin{equation}
\frac {dA}{dt} = c(n-6), \label{nminus6}
\end{equation}     
where $c$ is a material constant. When a cell with fewer than six sides shrinks
to a point, neighboring cells may change their number of sides   to maintain
the topological requirement that exactly three edges meet at at each  junction.  Therefore,
 a grain will typically change its number of sides, and therefore its  rate of
 growth, several times during the coarsening process.

Several physicists \cite{fra882,flyvbjerg1993model,marder1987soap,beenakker1986evolution}
used (\ref{nminus6})
as a starting point in  writing kinetic equations for densities $u_n(a,t)$
of cells having $n$
sides ($n$-gons) and area $a$ at time $t$. These take the form of constant convection
transport
equations with intrinsic flux terms, given by\begin{equation}
\label{eq:ke-big}
\partial_t u_n +c (n-6) \partial_a u_n = \sum_{l=2}^5  (l-6) u_l(0,t) \left(
\sum_{m=2}^M A_{lm}(t) u_m(a,t)\right), \quad n \ge 2. 
\end{equation}
Models
diverge in their choice of the matrix $A_{lm}$, which prescribes
mean field rules for how networks change topology when a grain vanishes.
In \cite{klobusicky2021two}, Menon, Pego, and the author presented a stochastic particle system, the  $M$-species model,  as an intermediate between kinetic equations and direct simulations of grain boundary coarsening.  The focus of  \cite{klobusicky2021two} was with well-posedness
of the limiting kinetic equations and simulations.  It remained, however, to provide
 estimates for convergence rates of the particle systems to their
hydrodynamic limits. A study was conducted in \cite{klobusicky2016concentration}
on a simplified model of one species, in which total loss of particles is  shown to be equivalent to a diminishing
urn process similar to Pittel's model of cannibalistic behavior \cite{pittel1987urn}.

In this paper, we build the groundwork for establishing concentration inequalities for the $M$-species model by restricting our attention to a model of two species.   Specifically, each particle
lives in $\mathbb R_+= [0, \infty)$ and is tagged as belonging to
either Species 1 or Species 2. Particles in Species 2 do not drift, while
particles in Species 1 drift at unit speed toward the origin.  When a particle
reaches the origin, it is removed, and a particle from Species 2 immediately mutates
into Species 1.
For a visual representation of the process in which mutations are represented
as  vertical jumps, see Fig. \ref{specfig}. The process just described can
be interpreted as a minimal model of network
coarsening, with   the behavior of particles in Species 1  analogous to the constant area decrease of cells with fewer
than six sides.
The removal of Species 1 and mutation of Species 2 are similar to the vanishing
of faces
and subsequent reassignment of neighboring cell topologies. 

The hydrodynamic limits for  densities $f_j(x,t)$ of particles in Species $j$ at position $x>0$ and time $t\ge 0$ are transport
equations with nonlinear intrinsic source terms, with\begin{align}
\partial_t f_1(x,t)- \partial_x f_1(x,t) &= \frac{f_1(0,t)f_2(x,t)}{N_2(t)},\label{kineq1}\\
\partial_t f_2(x,t) &= -\frac{f_1(0,t)f_2(x,t)}{N_2(t)},\label{kineq2}\\
f_1(x,0) = \bar f_1(x), \quad &f_2(x,0) = \bar f_2(x),\label{inits} 
\end{align}
where  $N_j(t) = \int_0^\infty f_j(x,t)dx$ is the total number of Species $j$. To allow for nondifferentiable initial conditions, we will exclusively work with  the integral form of (\ref{kineq1})-(\ref{inits}), written with Duhamel's formula as  
\begin{align}
f_1(x,t) &= \bar f_1(x+t)+\int_0^t f_2(x+t-s,s)\frac{f_1(0,s)}{N_2(s)}ds,
\label{kinint1}\\
f_2(x,t) &= \bar f_2(x)-\int_0^t f_2(x,s)\frac{f_1(0,s)}{N_2(s)}ds. \label{kinint2}
\end{align}

The main result for this paper is the convergence of empirical measures of the particle system to limiting kinetic
equations (\ref{kinint1})-(\ref{kinint2}). In Section \ref{sec:themodel}, we give a rigorous description of  the particle
system as a piecewise determinstic Markov process (PDMP),  lay out a deterministic discretization of the kinetic equations, and present the main results. Section \ref{disc} gives
a proof for Theorem \ref{detscheme}, which shows that the  discretization
converges to the kinetic equations at rate $\mathcal O(\delta+\omega(\delta,0))$, where $\delta$
is both the spatial and temporal step size of the  scheme, and $\omega(\delta,
0)$ is the modulus of continuity in the initial conditions. This is achieved through writing recurrence relations which compare total numbers restricted to intervals of size $\delta$.   
Section \ref{part} is a proof  of  Theorem \ref{finalh}, an exponential concentration inequality
(with respect to the initial total number of particles) between   the discretization
and particle system.  This involves similar recurrence inequalities seen in Section \ref{disc}, but now with an added task of showing that the inequalities occur under high probability.  We will need  to apply a generalization  of Hoeffding's inequality \cite{hoeffding1994probability}, 
a fundamental concentration inequality for sampling
without replacement, for establishing estimates on the total number of mutations occurring in intervals.    

Theorems \ref{detscheme}
and \ref{finalh} can be combined to produce our main result, Theorem \ref{majorthm},
which gives a concentration inequality between  empirical measures and
solutions of the kinetic equation. For sufficiently small $\varepsilon > 0$ and $n>n(\varepsilon)$,  the inequality takes the form
\begin{align}
&\mathbb P_n\left(\sup_{t\le T'}d((\mu_1(t), \mu_2(t)), (\mu^n_1(t), \mu^n_2(t)) ) \ge \varepsilon\right)
\le  \frac
{C} {\varepsilon^2}\exp(-\tilde C\varepsilon^5n).
\end{align}
Here, for $j = 1,2$ and time $0\le t\le T'$,   $ \mu^n_j(t)$ is the $n$-particle empirical measure for  positions of Species $j$ , and $\mu_j(t,dx) = f_j(t,x)dx$ where $f_j(x,t)$ is the solution of (\ref{kinint1})-(\ref{kinint2}). The metric $d$ is a sum of Kolmogorov-Smirnov metrics between measures of each species.  The constants $C$, $\tilde C$, and  $T'$  all depend on the initial conditions $\bar f$.   

We conclude with  Section \ref{prooftheo1}, in which we
derive an explicit solution of the kinetic equations and prove the well-posedness stated in Theorem \ref{solution}, relying on several well-known facts from renewal theory.
 We stress that explicit solutions are not used  in either the  proofs of Theorems
\ref{detscheme} and \ref{finalh}, as we hope to extend the methods
used here to  the $M$-species
 model which have no known
explicit solutions.

\section{Particle model and statement of results} \label{sec:themodel}

\begin{figure}
\begin{centering}
\includegraphics[width=.7\textwidth]{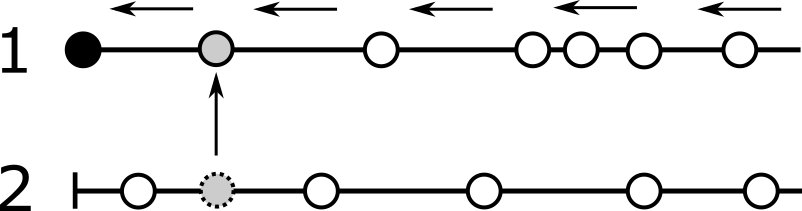}\label{specfig}
\caption{\textbf{A schematic  of the particle system}. Particles in Species
1 (top line) drift toward the origin at unit speed. When a particle (shown in black) reaches the origin,
it is removed. Another particle (shown in grey, with dashed outline immediately before mutation and solid outline after) is randomly selected from Species 2 (bottom line) to mutate into Species
1.   }
\end{centering}
\end{figure}

\subsection{A two-species particle system and its kinetic limit} \label{partsyskin}

We now  formally define  the stochastic process $\{X^n(t)\}_{t\ge 0}$  for an initial
system of $n$ particles. Each particle lives in one of two ordered copies
of $\mathbb R_+ = (0,\infty)$, which
we refer to as Species 1 and Species 2.
Since particles may be removed during the process,  the state space  $E^n$
consists of states  \begin{equation}
\mathbf x = \{(x_i,s_i): i =  1, \dots, |\mathbf
x|,\quad  |\mathbf x| \le n\}, \quad   
\end{equation}
with particle locations $x_i \in [0,\infty)$ and   labels 
$s_i\in \{1,2\}$ denoting each particle's species. The state space can be
expressed as a disjoint union of positive orthants
\begin{equation}
E^n = \coprod_{l+m\le n} E^n_{(l,m)},  \label{pdmpstates}
\end{equation}
with  $E^n_{(l,m)}= \mathbb R_+^l \times \mathbb R_+^m$ denoting positions
 for $l$ particles in Species 1 and $m$ particles in Species 2.

 Fix  an initial state $X^n(0) = \{(x_1^0,s_1^0), \dots,
(x_n^0,s_n^0)\} \in E^n$, and denote  $\alpha$  as an index for a particle in Species
1 closest to the origin, meaning  $x_\alpha^0 \le x_i^0$ for $i = 1, \dots,
 n$ and $s_\alpha^0 = 1$.  Now let $\tau_1 = x_\alpha^0$ denote the time
until a particle reaches the origin.  Define
$X^n(t) \in E^n$ for $t \in [0, \tau_1)$ deterministically by advecting particles
in Species 1 toward the origin at unit speed while keeping particles in Species
2 fixed:    \begin{equation}
 s_i(t) = s_i^0, \quad 
x_i(t) = \begin{cases}x_i^0-t, & s_i^0 = 1, \\
x_i^0, & s_i^0 = 2 \\
\end{cases} ,
\qquad i = 1, \dots, n.
\end{equation} 
 Randomness is introduced with a mutation at time $t =   \tau_1$.  At this time,
the smallest
particle in Species 1 has reached the origin, and  is removed from the system.
Furthermore, a particle $(x_\iota(\tau_1^-), 2)$ selected with uniform probability
from Species 2 mutates while keeping its position, meaning 
\begin{equation}
(x_\iota(\tau_1), s_\iota(\tau_1)) = (x_\iota(\tau_1^-), 1).
\end{equation}
Finally, particle indices $ i > \alpha$ decrement by one so that the  index
set is \{1, \dots, $n-1\}$.
The (now stochastic)
process then repeats deterministic drift until a particle from Species 1
reaches the origin
at some time $\tau_2$, again triggering a random mutation, and the process
continues  until there are no particles left in Species 2. For instances
in which  multiple  particles reach the origin simultaneously, particles
in    Species 2 are selected to mutate by sampling without replacement. The process $\{X(t)\}_{t \ge 0}$ is an example of the general $M$-species model, which  was shown in \cite{klobusicky2021two}  to be  a class of piecewise deterministic Markov processes (PDMPs).  The stochastic process induces a filtered probability space $(\Omega, \mathcal F, \mathbb P_n, \{\mathcal F(t)\}_{t\ge 0})$, where $\mathcal F$ is the natural filtration. Davis \cite{dav84} established that PDMPs are in fact strong Markov, which we will use in Section \ref{part} when considering events before and after certain mutation times.

At time $t \ge 0$, denote the number  of particles in Species $j$
as $N_j^n(t)$, and the total number as $N^n(t) = N_1^n(t)+N_2^n(t)$. Empirical
densities for species with $n$ initial particles are then defined as
\begin{equation}
\mu^n_j(t)= \frac 1n \sum_{i= 1}^{N^n(t)} \delta(x_i)\cdot \mathbf 1(s_i
= j), \quad
\quad j = 1,2.
\end{equation}

The differential form of the limiting equations of the infinite particle limit $n\rightarrow
\infty$ are given by equations  (\ref{kineq1})-(\ref{inits}), in which for each species $j$ the limiting measures
$\lim_{n\rightarrow \infty}\mu_j^n\rightarrow \mu_j$ are deterministic with densities $f_j(x,t).$
We will require that that  $N_1(0)+N_2(0) = 1$ and $N_2(0)>0$. The left hand sides of (\ref{kineq1}) and (\ref{kineq2})
represent the constant drift of  Species 1 toward the origin and zero drift in
Species 2. The right hand sides give the intrinsic flux arising from mutations
selected from a  normalized density of $f_2$, occurring at a  frequency of
 $f_1(0,t)$.
To allow for nondifferentiable initial data and solutions, we will use
 the integral form   (\ref{kinint1})-(\ref{kinint2}) with initial data $(\bar f_1, \bar f_2) \in Z^2$,
 where $Z$ denotes the cone  of positive, continuous, and locally bounded
functions under the $L^1(\mathbb R_+)$ norm topology.
Equations   (\ref{kinint1})-(\ref{kinint2}) reach a singularity when $N_2
= 0$, corresponding to when there are no more Species 2 particles to mutate.
This occurs at time
\be
  T(\bar f) = \sup_{t>0}\{N_2(t)>0\}.
\ee
  
The derivation of an explicit solution of  (\ref{kinint1})-(\ref{kinint2}) first relies on explicitly solving for  the
removal rate, which we write as $a(t) = f_1(0,t)$, and
subsequently the total loss 
\begin{equation}
L(t) = \int_0^ta(s)ds,
\end{equation}which may be interpreted  as an ``internal clock" to the
system, counting normalized total visits to the origin or, equivalently, total mutations.
                
\begin{theorem}\label{solution}
Let $\bar f = (\bar f_1, \bar f_2) \in Z^2 $ with $N_2(0)>0$, 

 (a) The removal rate $a(t) = f(0,t)$ and  $N_2(t)$ may be written in terms of the initial conditions as
\begin{align}\label{aexp}
a(t) &= \sum_{j = 0}^\infty \left(\frac{\bar f_2}{N_2(0)}\right)^{*(j)}(t)*\bar
f_1(t),\\
N_2(t) &= N_2(0)-L(t).\label{nform} \end{align}
Here, an exponent of $*(j)$ denotes $j$-fold self convolution (with $f^{*(0)}
=1$). For $0\le t<T(\bar f)$, the
solution $(f_1(x,t), f_2(x,t)) \in Z^2$  of  (\ref{kinint1})-(\ref{kinint2})
is unique 
and has the explicit form\begin{align}
f_1(x,t)&= \bar f_1(x+t)+\int_0^t\frac{\bar f_2(x+t-s)}{N_2(0)}a(s)ds,\label{f1exp}\\
f_2(x,t) &= \frac{N_2(t)}{N_2(0)}\bar f_2(x).  \label{f2exp}
\end{align}

(b) For $\bar f_1, \bar f_2 \in  Z$ and  $0\le t < T(\bar f)$,  (\ref{aexp})-(\ref{f2exp})
defines a continuous dynamical system in $Z^2$, so that the map $(\bar f_1,\bar
f_2,t)\mapsto (f_1(\cdot,t),f_2(\cdot ,t))$ is in $C(Z^2\times [0,T(\bar
f)], Z^2)$.
 \end{theorem}

The proof for Theorem \ref{solution} is postponed until Section \ref{prooftheo1},
as neither
the explicit solution nor its derivation will be used in future results. The well-posedness of (\ref{kinint1})-(\ref{kinint2}) is invoked for defining constants in the convergence analysis of Sections \ref{disc} and \ref{part}. Note, however, that the particle system in this paper  is a special case of the $M$-species model developed in \cite{klobusicky2021two}, in which  well-posedness is derived through a Banach fixed point argument, rather than appealing to
an explicit solution. 

\subsection{Discretization scheme of kinetic equations}

\label{discdetoutline}
To enable us to write down recurrence inequalities involving total numbers restricted to an interval, we will 
 construct a deterministic scheme for  (\ref{kinint1})-(\ref{kinint2}).  We do so with measures $\tilde \mu_1(t, \cdot;\delta),\tilde \mu_2(t, \cdot;\delta) \in \mathcal{M}(\mathbb
R_+)$ at time $t>0$ which are piecewise constant in  $\delta>0$ sized time intervals $\Delta t_k = [\delta (k-1), \delta k)$ for $k \ge 1$. Note that while these measures depend on $\delta$, we will often suppress this argument in the notation for simplicity in presentation.

Initial measures are given by  
\begin{equation}
\tilde \mu_j(t,\cdot) = \bar \mu_j, \quad t \in [0, \delta), \quad j = 1,2,
\label{discdetinit} \end{equation}
with the requirement that $(\bar \mu_1+\bar \mu_2)([0, \infty)) = 1$.  
For each time step  $t_k = k\delta$, we define the incremental loss  over
a time interval as
\begin{equation}\label{diffl}
\Delta \tilde L(t_k) =\begin{cases}0, & k = 0,  \\
\tilde \mu_1(t_{k-1},[0,\delta)), &  k \ge 1. \\
\end{cases} \quad   
\end{equation}
Total number for Species 2 then decreases by  the incremental loss, with

\begin{equation}
\tilde N_2(t_k) =\begin{cases} \tilde \mu_2(0, [0, \infty)), & k = 0, \\
 \tilde N_2(t_{k-1})-\Delta\tilde L(t_{k}), & k \ge 1. \\
\end{cases} \quad 
\end{equation}
Measures update  by  a shift of  distance $\delta$ toward the origin in Species
1 followed by mutation in Species 2  of total number $\Delta\tilde
L(t_{k})$.
Therefore, for $t \in [t_{k},t_{k+1})$, $k \ge 1$, we update with 
\begin{align}
\tilde \mu_1(t) &=  S_{\delta}(\tilde \mu_2(t_{k-1}))+\frac{\Delta\tilde
L(t_{k})}{\tilde
N_2(t_{k-1})}\tilde \mu_2(t_{k-1}),\label{updatediscneg1}\\
\tilde \mu_2(t) &=  \tilde \mu_2(t_{k-1})\left(1-\frac{\Delta\tilde L(t_{k})}{\tilde
N_2(t_{k-1})}\right). \quad
 \label{updatedisc0}
 \end{align}
Here $S_h$ is the left translation operator acting on measures, defined through
the cumulative
function $F_\mu$ of a measure $\mu \in \mathcal M(\mathbb R_+)$ by 
\begin{equation}
S_h (F_\mu(x)) =
F_\mu(x+h)-F_\mu(h), \quad h \ge 0 .
\end{equation}
Since Species 1 shifts before adding
mutated particles from Species 2, we have a conserved quantity $\tilde N_1(t)
= \tilde N_1(0)$, and thus the total number 
\begin{equation}
\tilde N(t_k):= \tilde N_1(t_k)+\tilde N_2(t_k) = 1-\sum_{i = 1}^{k-1}\Delta\tilde
 L(t_i).
\end{equation}This scheme remains well-defined as long as $ \tilde N_2(t)>0$.
  It is clear that   
\begin{equation}
\tilde N_2(t) \ge \tilde N_2(0)-(\bar \mu_1([0,t])+\bar \mu_2[0,t]), \label{n2tildefinite}
\end{equation}
so for initial measures in $Z$ we are always able to find a nonzero length interval
of existence   $[0,T_1(\bar f))$, with 
\begin{equation}
T_1(\bar f) = \sup\{t:\tilde N_2(t)>\tilde N_2(0)/2\}.
\end{equation}

\subsection{Main results: convergence rates and concentration inequalities}

Our first main result gives a comparison between the deterministic discretization
and solutions of (\ref{kinint1})-(\ref{kinint2}) using the Kolmogorov-Smirnov
(KS) metric. For two measures
$\nu, \eta \in \mathcal M(\mathbb R_+), $ with associated cumulative functions
$F_{\nu}(x) = \nu([0,x])$ and $F_{\eta}(x) = \eta([0,x])$ the KS metric is
defined as
\begin{equation}
d_{KS}(\nu, \eta) =  \sup_{x \in \mathbb R}|F_{\nu}(x)-F_{\eta}(x)|.
\end{equation}
For handling convergence of both species, we define a metric on $\mathcal
M(\mathbb R_+) \times \mathcal M(\mathbb R_+)$ between $\nu =  (\nu_1, \nu_2)$ and $\eta = (\eta_1,\eta_2)$ by
\begin{equation}
d( \nu,\eta) =d_{KS}(\nu_1, \eta_1) + d_{KS}(\nu_2,
\eta_2).   \label{ourmetric}
\end{equation}  
As we are often working with measures, we define measures associated
to solutions of (\ref{kinint1})-(\ref{kinint2}) by
\begin{equation}
\mu_j(t,dx) = f_j(x,t)dx \quad j = 1,2.
\end{equation}
We will also need to track the modulus of continuity for solutions.
For  densities $(f_1(x,t), f_2(x,t))$, we let \begin{equation}\label{moc}
\omega(\delta, t) = \sup_{x \in  \mathbb R_+} \sum_{j = 1}^2|f_j(x+\delta,t)-f_j(x,t)|.
\end{equation}
We will impose that  initial conditions  have compact support for the sake
of clarity, as  calculations relating  convergence and the decay of initial
conditions can become  rather technical. Since initial conditions are continuous,
compact support implies that $\omega(\delta, 0) \rightarrow 0$ as $\delta\rightarrow
0$.

 \begin{theorem} \label{detscheme}
 
 Let $ \bar f = (\bar f_1, \bar f_2) \in Z^2$ have compact support with $N_2(0)>0$. For $t< T(\bar f)$, let   $\mu(t) = (\mu_1(t), \mu_2(t))$  be measures for
the unique solution to (\ref{kinint1})-(\ref{kinint2}). Let measures $\tilde \mu(t ) = (\tilde \mu_1(t ), \tilde \mu_2(t))$ for the discretization scheme with  step size of $\delta>0$ also have initial conditions $\bar f$. Then there exist positive constants $\delta^d$ and $C^d$ such that for  all $\delta \in (0, \delta^d)$ and $T' \in [0,T(\bar f))$,
 \begin{equation}
 \sup_{t \in [0, T']}d(\tilde \mu(t ),\mu(t))\le  C_d( \delta+ \omega(\delta,0)). \label{detthm2}
 \end{equation} 
 The constants $C^d$ and $\delta^d$ are dependent on $L^\infty$ and $L^1$ bounds of solutions in Theorem \ref{solution}, and the compact support bound  $M = \sup\{x:\sum_{j = 1,2}\bar f_{j}(x) >0\}$.  
\end{theorem}

 To set up for the next main result, we generate initial conditions $\bar \mu^n = (\bar \mu^n_1,  \bar \mu^n_2)$
for the particle system with uniform spacing through the cumulative distribution functions $F_j(x) = \mu_j([0,x])$ for $j = 1,2$. The explicit  particle positions are
\begin{align}
( x_i(0), s_i(0)) = \begin{cases}(F_1^{-1}(i/n),1) & i = 1, \dots,
\lfloor nN_1(0) \rfloor, \\
(F_2^{-1}(( nN_1(0)-i)/n),2) & i = \lfloor nN_1(0) \rfloor+1, \dots, n, \\
\end{cases} \label{stochassign}
\end{align}
where $F^{-1}$ is the quantile function of a cumulative function $F$. 
For initial conditions
$(\bar f_1, \bar f_2)\in Z^2$ with $\bar \mu_j(dx) = \bar f_jdx$, it is easy to show that  $d(\bar
\mu^n,  \bar \mu) \rightarrow 0$ in law as $n\rightarrow \infty$, and that  there exists a positive
integer $n_{0}(\bar f)$ such that if $n> n_{0}(\bar f)$, 
\begin{equation}
\bar \mu_1^n([0,T_1(\bar
f)])+\bar \mu_2^n([0,T_1(\bar f)]) \le 2N_2^n(0)/3 \quad \hbox{ for } n>n_{0}(\bar
f), \label{nfcond} 
\end{equation}
meaning that there is always a particle available to mutate at each jumping time
in $[0, T_1(\bar f)]$, and the process  therefore does not reach its cemetery
state.

The next theorem, which we show in Section \ref{part},
 gives an exponential concentration inequality between the deterministic discretization and the particle system.

 \begin{theorem} \label{finalh}
 Let $ \bar f = (\bar f_1, \bar f_2) \in Z^2$ have compact support with $N_2(0)>0$. For $t<
T(\bar f)$, let $\mu^{n}(t) = (\mu_1^{n}(t), \mu_2^{n}(t))$  be empirical measures for $X^n(t)$  generated from (\ref{stochassign}) with $\bar f \in Z^2$. Let measures $\tilde
\mu(t ) = (\tilde \mu_1(t ), \tilde \mu_2(t))$ for the discretization scheme
with  step size of $\delta>0$ have initial conditions $\bar f$. Then there exist positive constants $\delta^p, C^p, C^p_2,C_3^p$ such that for all $\delta\in (0,\delta^p)$
there exists   $n^p(\delta)>0$ such that for all integers $n>n^p(\delta)$ and $T'<T(\bar f),$ 
\begin{equation}
\mathbb P_n\left(\sup_{t\le T'}d(\tilde \mu(t ),\mu^n(t))\ge C^{p}(\delta+\omega(0,\delta)\right) \le \frac
{C^{p}_2} {\delta^2}\exp(-C^{p}_3\delta^5n).  \label{detthm3}
\end{equation}
The constants $\delta^p, C^p, C^p_2,C_3^p$ are all dependent on $L^\infty$ and $L^1$
bounds of solutions in Theorem \ref{solution} with initial conditions of
$\bar f = (\bar f_1, \bar f_2)$, and the compact support bound  $M = \sup\{x:\sum_{j
= 1,2}\bar f_{j}(x) >0\}$.  
\end{theorem}

From Theorems \ref{detscheme} and \ref{finalh}, it is straightforward to obtain our main concentration inequality for initial conditions which are also locally Lipschitz.

\begin{theorem} \label{majorthm} 
Let $\bar f = (\bar f_1, \bar f_2) \in Z^2$ with   $N_2(0)>0$ be both compactly supported and  locally Lipschitz, so that  $\omega(0,\delta) \le   C^{\omega}
\delta$. For   $C^p$ and $ \delta^p$ determined from Theorem \ref{finalh}, let $\varepsilon \in (0,2C^pC^{\omega}
\delta^p)$. Then there exist positive constants $C, \tilde C$ and    $N(\varepsilon)>0$ such that for all integers $n>N(\varepsilon)$ and $T'<T(\bar f),$
\begin{align}
&\mathbb P_n\left(\sup_{t\le T'}d(\mu(t), \mu^n(t) )\ge \varepsilon\right)
\le  \frac
{C} {\varepsilon^2}\exp(-\tilde C\varepsilon^5n). \label{mainconc}
\end{align}
  
\end{theorem}

\begin{proof}
Since initial conditions are locally Lipschitz and compactly supported, we may replace the $C^d( \delta+ \omega(\delta,0))$ and $C^p( \delta+ \omega(\delta,0))$ terms in  Theorem \ref{detscheme} and with \ref{finalh} with $\hat C^d \delta$ and $\hat C^p \delta$, respectively,  where $\hat C^d =  C^dC^\omega$ and $\hat C^p = C^pC^\omega$.

Let $\varepsilon\in (0,2\hat C^p\delta_p)$  and choose $\delta = \varepsilon /(  2\hat C^d)$. We then have
\begin{align}
&\mathbb P\left(\sup_{t\le T'}d(\mu(t), \mu^n(t)) \ge \varepsilon\right) \le    \mathbb P\left(\sup_{t\le T'}d(\tilde
\mu(t),
\mu^n(t))+d(\mu(t),
\tilde \mu(t)) \ge \varepsilon\right)\\
& \le \mathbb P\left(\sup_{t\le T'}d(\tilde
\mu(t),
\mu^n(t))\ge \varepsilon-\hat C^d\delta\right) = 
\mathbb P\left(\sup_{t\le T'}d(\tilde
\mu(t),
\mu^n(t))\ge \varepsilon/2\right) \nonumber\\
&\le  \frac
{4(\hat C^p)^2C_{2}^p} {\varepsilon^2}\exp\left(-\frac{C_{3}^p}{32(\hat C^p)^5}\varepsilon^5n\right).\nonumber
\end{align}
The second inequality uses Theorem \ref{detscheme}, and the third uses Theorem \ref{finalh}. We obtain (\ref{mainconc}) with $C = 4(\hat C^p)^2C_{2}^p$ and $\tilde C =  C_{3}^p/(32(\hat C^p)^5)$.
\end{proof}
From Theorem \ref{majorthm},  an application of the Borel-Cantelli lemma then gives us a strong law of large numbers. 

\begin{cor}
Under the product measure $\mathbb Q = \prod_{n\ge 2} \mathbb P_n$, for $T' < T(\bar f)$,   \begin{equation}
  \lim_{n\rightarrow 0} \sup_{t \le T'}d( \mu^{n}(t), \mu(t)) = 0 \quad \hbox{ almost surely.}
\end{equation} \end{cor}

\section{Comparison of kinetic equations and deterministic scheme}\label{disc}

In this section, we present a proof of Theorem \ref{detscheme}. The discretization
scheme outlined in Section \ref{discdetoutline} allows
us to write recursive formulas at time steps $t_k$ related to measures restricted
to size $\delta$ intervals, which we denote as  
\begin{equation}
I_l = [(l-1)\delta, l\delta), \quad l \ge 1.
\end{equation}In Section \ref{regdisc} we collect estimates related to growth
of quantities for solutions of the kinetic equations and the discretization
scheme, including the modulus of continuity $\omega(\delta, t)$ and total
number contained in an interval.  Estimates related to the comparison between
solutions of the kinetic equations and the discretization are presented 
in Section \ref{convest}.  The main quantity of interest is  the difference
of total number in intervals $I_l$ at times $t_k$.  Through constructing
a closed recurrence inequality,  we show differences are of order $\delta^2+\delta
\omega(\delta, 0)$. In Section \ref{poofdettheo}, these differences are then
summed over $[0,M]$
to establish Theorem \ref{detscheme}.

\subsection{Growth estimates}\label{regdisc}



Our estimates for solutions of (\ref{kinint1})-(\ref{kinint2}) and iterations
(\ref{updatediscneg1})-(\ref{updatedisc0}) will make
frequent use of the constant bounds\begin{align}
C_{\infty}( \bar f) &= \max_{j = 1,2}\sup_{s \in [0,T_1(\bar f)]} \|f_j(x,s)\|_{\infty}, \label{cinf}\\
C_b( \bar f)&=\max \{1/N_2(T_1(\bar f)), 1/\tilde N_2(T_1(\bar f))\},  \label{cb}
\end{align}
which are dependent upon the initial
conditions $\bar f = (\bar f_1, \bar f_2)$ and time of existence $T(\bar
f)$.  That $C_{\infty}( \bar f), C_b(\bar
f)$ are finite follows from well-posedness of $(f_1(x,t), f_2(x,t))$ in Theorem
\ref{solution} and the existence of $T_1(\bar f)$ established from (\ref{n2tildefinite}).
For simplicity, in future estimates we will refer to these constants simply
as $C_\infty$
and $C_b$. As we will see in Lemma \ref{mbounds}, it  will be necessary to
further restrict our interval of existence  to   $t \in [0, T_2(\bar f)])$,
with
\begin{equation}
T_2(\bar f) = T_1(\bar f) \wedge 1/(8C_\infty C_b).
\end{equation}
We will work with solutions $(f_1(x,t), f_2(x,t))$ of (\ref{kinint1})-(\ref{kinint2})
having initial conditions $\bar f_1,\bar
f_2 \in Z$ with compact support in some interval $[0,M] \subset \mathbb R_+$.
For the deterministic scheme, measures
 have
identical initial conditions as the kinetic equations, with  $\bar \mu_j(dx)
= \bar f_j(x)dx$. 

We begin  with a simple estimate on the propagation of the  modulus of continuity.

\begin{lemma}\label{moclem} For all $\delta>0$ and $t \le T_2(\bar f)$, 
\begin{align}
\omega(\delta, t) \le C_1\omega(\delta, 0), \label{moc1}
\end{align}
where $C_1 = 2\exp(2C_b).$
\end{lemma}
\begin{proof}
Since $N_2(t)$ is decreasing,  we use  (\ref{kinint1})-(\ref{kinint2}) to
show
\begin{align}
&\sum_{j = 1}^2|f_j(x+\delta,t)-f_j(x,t)|\label{mocinq}\\&\le|f_1(x+\delta+t,0)-f_1(x+t,0)|+|f_2(x+\delta,0)-f_2(x,0)|\nonumber\\
&+C_b \int_0^t|f_2(x+\delta+t-s,s)-f_2(x+t-s,s)|+|f_2(x+\delta,s)-f_2(x,s)|dL(s),\nonumber
\end{align}
where $L(s) = \int_0^t f_1(0,s)ds$ is the total loss. By taking the supremum
of the above inequality  over $x\in \mathbb R_+$,
from the definition of $\omega(\delta, t)$ given in (\ref{moc}),\begin{align}
\omega(\delta,t) \le 2\omega(\delta,0)+ 2C_b\int_0^t\omega(\delta,s)
dL(s). 
\end{align}
From Gronwall's inequality, 
\begin{equation}
\omega(\delta,t) \le2\exp(2C_bL(t))\omega(\delta,0). 
\end{equation}
Since $L(t) \le 1$, we obtain (\ref{moc1}).
\end{proof}

Next, we turn to studying the maximum total number  of a measure on
length $\delta$ intervals, denoted  as  \begin{equation}
m^j_\delta(t) = \sup_{\substack{I:|I|=  \delta}} \mu_j(t,I),
\quad j= 1,2, \qquad
 m_\delta(t) =\sum_{j
= 1}^2 m^j_\delta(t).  \label{mdef}
\end{equation}
 We define $\tilde m_\delta(t)$ similarly.

\begin{lemma}\label{mbounds}For $\delta>0$ and $t<T_2(\bar f)$,
\begin{align}
 m_\delta(t)&\le C_2\delta ,\label{mlem1}\\
\tilde m_\delta(t_k)&\le \tilde C_2 \delta, \label{mlem2}\\
\tilde N(t_k)&\ge 1-\tilde C_2  \delta k . \label{mlem3}
\end{align}
With constants  
\begin{equation}
C_2 =2C_\infty \exp(C_b),   \quad  \tilde C_2 = \frac{8C_\infty^2}{1- 4C_bC_\infty
T_2(\bar f)}.
\end{equation}
\end{lemma}
\begin{proof}
To show (\ref{mlem1}), for an interval $I$ with $|I| \le \delta$, integrate
 (\ref{kinint1})-(\ref{kinint2})
over $I$ to obtain
\begin{align}
\mu_1(t,I)&= \mu_1(0,I+t) +\int_{0}^{t} \frac{\mu_2(s,I+t-s)}{N_2(s)}
dL(s),\\
\mu_2(t,I)&= \mu_2(0,I)-\int_{0}^{t} \frac{\mu_2(s,I)}{N_2(s)}
dL(s)
\end{align}
By taking  the supremum over all length $\delta$  intervals, we find\begin{align}
m_\delta^1(t) &\le m_\delta^1(0)+C_b\int_0^t  m_\delta^2(s)
dL(s), \label{m1less}\\
m_\delta^2(t) &\le m_\delta^2(0).\label{m2less}
\end{align}
We then obtain (\ref{mlem1}) by summing  (\ref{m1less})-(\ref{m2less}), applying
Gronwall's lemma, and observing from initial conditions that  $m_\delta(0)
 \le 2C_\infty \delta$. 

To show (\ref{mlem2}), we will work with
\begin{equation}
\hat m^j_\delta(t) = \sup_{l\ge 1} \tilde \mu_j(t,I_l),
\quad j= 1,2, \qquad
 \hat m_\delta(t) =\sum_{j
= 1}^2 \hat m^j_\delta(t).  
\end{equation}
Note  that 
\begin{equation}
\tilde  m^j_\delta(t) \le 2\hat m^j_\delta(t). \label{hatvstilde}
\end{equation} 

From evaluating measures on $I_l$,  the recursion (\ref{updatediscneg1})-(\ref{updatedisc0})
implies that for $l \ge 1$ and $s \in [t_k, t_{k+1})$,    \begin{align}
\tilde\mu_1(s,I_l) &= \tilde\mu_1(t_{k-1},I_{l+1})+\frac{\Delta\tilde L(t_{k})}{\tilde
N_2(t_{k-1})}\tilde\mu_2(t_{k-1},I_{l}), \label{recmu1}\\
\tilde\mu_2(s,I_l) &= \tilde\mu_2(t_{k-1},I_l)\left(1-\frac{\Delta\tilde
L(t_{k})}{\tilde
N_2(t_{k-1})}\right). \label{recmu2}
\end{align}
We now
take the supremum over $j\ge 1$ in (\ref{recmu1})-(\ref{recmu2}) and sum
to obtain
\begin{equation}\label{miter}
\hat m_\delta(t_k)\le \hat m_\delta(t_{k-1})+C_b\left(\hat m_\delta(t_{k-1})\right)^2.
\end{equation}

For simplicity, we rescale by writing $\hat m_\delta(t_k) = A(t_k)\delta$.
From  (\ref{miter}), and noting
\begin{equation}
 \hat
 m_\delta(0) \le  2\tilde m_\delta(0) = 2m_\delta(0) \le 4C_\infty \delta,
 \end{equation}
 we obtain the recurrence inequalities
\begin{align}\label{eschemea}
 A(t_k)&\le A(t_{k-1})+\delta C_b(A(t_{k-1}))^2 , \\
 A(0) &\le 4C_\infty. \label{eschemeb}
\end{align}
In the case of equality,  (\ref{eschemea})-(\ref{eschemeb}) is an Euler scheme
for the differential equation $g'(t) =  C_b g^2(t)$ with $g(0) =  C_\infty.$
We may check directly that $A(t_k) \le g(t_k)$  before  blowup, or
\begin{equation}  
A(t_k) <  \frac{A(0)}{1-kA(0) C_b\delta}  \label{aineq}
\end{equation}
for  $t_k \le T_2(\bar f)$. We then use (\ref{aineq}), (\ref{eschemeb}),
 and (\ref{hatvstilde}) to obtain  (\ref{mlem2}).

Finally, (\ref{mlem3}) follows immediately, since 
\begin{equation}
\tilde N_2(t_k) =1-  \sum_{i= 1}^{k} \Delta\tilde  L(t_i)\ge 1-  \sum_{i=
1}^{k} \tilde m_\delta(t_i).
\end{equation}
\end{proof}

 We finish this subsection with one more estimate related to the  incremental
loss and total number in the kinetic limit.
\begin{lemma} \label{nllem} For $t_k < T_2(\bar f),
 $ \begin{align}
&\Delta L(t_k): = L(t_k)-L(t_{k-1}) \le C_2 \delta+C_bC_{\infty}^2 \delta
^2, \label{ldiff}\\
&N(t_k)
 \ge 1-C_2k\delta -C_bC_{\infty}^2k\delta^2.
\label{nbound} \end{align}
\end{lemma}
\begin{proof}
We use (\ref{kinint1}) on the removal rate  $f_1(0,t)$ to obtain
\begin{align}
\Delta L(t_k) &= \int_{0}^\delta f_1(0,t_{k-1}+s)ds\\
 &= \int_0^\delta \left(f_1(s,t_{k-1})+ \int_0^s \frac{f_1(0,t_{k-1}+r)}{N_2(t_{k-1}+r)}
f_2(s-r,t_{k-1}+r)dr\right)ds\nonumber \\
&\le \mu_1(t_{k-1},[0,\delta])+C_bC_{\infty}^2\delta^2 .\nonumber 
\end{align}
From (\ref{mlem1}), we obtain (\ref{ldiff}). Since $
N(t_k) = 1-\sum_{i = 1}^{k} \Delta L(t_i),$  (\ref{nbound}) also follows.
\end{proof}

\subsection{Convergence estimates} \label{convest}

We now use estimates from the previous subsection to establish asymptotics
for the differences of total number between the
solution of (\ref{kineq1})-(\ref{inits}) and its  discretization.
We begin with a simple  result which follows immediately from Lemma \ref{nllem}
comparing incremental losses and total numbers of species.
\begin{cor} \label{difflncor} For $t_k < T_2(\bar f),
$  \begin{align}
|\Delta L(t_k)-\Delta \tilde L(t_k)| &\le |\mu_1(t_{k-1},[0,\delta])-\tilde
\mu_1(t_{k-1},[0,\delta])|+C_bC_{\infty}^2\delta^2. \\
|N(t_k)- \tilde N(t_k)| &\le \sum_{i = 1}^{k}|\mu_1(t_{i},[0,\delta])-\tilde
\mu_1(t_{i},[0,\delta])|+T(\bar f)C_bC_{\infty}^2\delta. \label{diffncorr}
\end{align}
\end{cor}

To compare behavior on an interval $I_j$, we will use a formula similar to
 (\ref{recmu1})-(\ref{recmu2}) for evolving
densities over a time step $\Delta
t_k = [t_{k-1}, t_k)$. It follows directly from (\ref{kinint1})-(\ref{kinint2})
that      
 \begin{align}
f_1(x,t_k)&=  f_1(x+\delta, t_{k-1})+\int_{\Delta t_k}\frac{ f_2(x+t_k-s,s)}{N_2(s)}f_1(0,s)ds,\label{maineqn1}\\
 f_2(x,t_k) &=f_2(x,t_{k-1})- \int_{\Delta t_k} \frac{f_2(x,s)}{N_2(s)} f_1(0,s)ds.\label{maineqn2}
\end{align}To arrive at  an  estimate for the difference of total number
in an
interval, we use (\ref{updatediscneg1})-(\ref{updatedisc0}) and (\ref{maineqn1})-(\ref{maineqn2})
to express the difference of total number of an interval
in Species 1 as
\begin{align}\label{fdiff}
&|\mu_1(t_k,I_l)-\tilde \mu_1(t_k,I_l)|\le |\mu_1(t_{k-1},I_{l+1})-\tilde
\mu_1(t_{k-1},I_{l+1})|\\ &+\left|\int_{I_l}\int_{\Delta t_{k}} \frac{f_2(x+t_{k}-s,s)}{N_2(s)}
dL(s)dx-\frac{\Delta \tilde L(t_{k})}{\tilde N_2(t_{k-1})} \tilde
\mu_2(t_{k-1},I_{l})\right|.\nonumber 
\end{align}
For Species 2, 
\begin{align}
|\mu_2(t_k,I_l)-\tilde \mu_2(t_k,I_l)|\le |\mu_2(t_{k-1},I_{l})-\tilde \mu_2(t_{k-1},I_{l})|\label{f2diff2}\\
+\left|\int_{I_l}\int_{\Delta t_{k}} \frac{f_2(x,s)}{N_2(s)}
dL(s)dx-\frac{\Delta \tilde L(t_{k})}{\tilde N_2(t_{k-1})} \tilde
\mu_2(t_{k-1},I_{,})\right|.\nonumber 
\end{align}
The following lemma allows us to estimate the integrals in (\ref{fdiff})
and (\ref{f2diff2}) by quantities at discretized times $t_k$.

\begin{lemma} \label{intfest}
For $ l\ge 1$ and $t_k < T_2(\bar f),$ 
 \begin{align}
\left|\int_{I_l}\int_{\Delta t_{k}} \frac{f_2(x+t_{k}-s,s)}{N_2(s)}
dL(s)dx-\frac{\Delta  L(t_{k})}{ N_2(t_{k-1})} 
\mu_2(t_{k-1},I_{l})\right|  
 &=  \mathcal O(\delta^3+\delta^2\omega(\delta,0)). \label{tri02} \\
 \left|\int_{I_l}\int_{\Delta t_{k}} \frac{f_2(x,s)}{N_2(s)}
dL(s)dx-\frac{\Delta  L(t_{k})}{ N_2(t_{k-1})} 
\mu_2(t_{k-1},I_{l})\right| &=\mathcal O(\delta^3). \label{tri03} 
\end{align}
\end{lemma}
\begin{proof}
We show (\ref{tri02}).  The proof for (\ref{tri03}) is similar. Denote the
left hand side of (\ref{tri02}) as $\mathcal I$.  Through the triangle inequality,
we have\begin{align}
&\mathcal I = \left|\int_{I_l}\int_{\Delta t_{k}} \frac{f_2(x+t_{k}-s,s)}{N_2(s)}
-\frac{f_2(x,t_{k-1})}{N_2(t_{k-1})}
 dL(s)dx\right| \label{tri0}\\
&\le\int_{I_l} \int_{\Delta t_{k}}\left| \frac{f_2(x+t_{k}-s,s)-f_2(x,t_{k-1})}{N_2(s)}
\right|dL(s)dx\nonumber\\
&+\int_{I_l} \int_{\Delta t_{k}}\left| f_2(x,t_{k-1})\left(\frac{1}{N_2(s)}-\frac{1}{N_2(t_{k-1})}
\right)\right|dL(s)dx\nonumber\\
&:= \mathcal I_1+\mathcal I_2.\nonumber
\end{align}

We first show that 
\begin{equation}
\mathcal I_1 \le C_b^2C_{\infty}^2C_2\delta^3+C_1C_2C_b\omega(\delta,0)\delta^2
+\mathcal O(\delta^4+\omega(\delta,0)\delta^3
).   \label{int1est}
\end{equation}

 This is done by another  use of the triangle inequality to align arguments
in time and space, 
\begin{align}
\mathcal I_1
\le C_b \Big(\int_{\Delta
t_{k}}\int_{I_l}|f_2(x+t_{k}-s,t_{k-1})-f_2(x,t_{k-1})|dxdL(s) \label{int11}\\
+\int_{I_l}\int_{\Delta t_{k}}|f_2(x+t_{k}-s,s)-f_2(x+t_{k}-s,t_{k-1})|dL(s)dx.\Big)
\nonumber
\end{align}
From (\ref{moc1}),(\ref{mlem1}) and (\ref{ldiff}), the first integral may
be bounded by
\begin{align}
\int_{\Delta
t_{k}}\int_{I_l}|f_2(x+t_{k}-s,t_{k-1})-f_2(x,t_{k-1})|dx dL(s)\label{flem1}\\ \le C_1C_2\omega(\delta,0)\delta^2+\mathcal
O(\omega(\delta,0)\delta^3).\nonumber 
\end{align}
For the last integrand for (\ref{int11}), we may use (\ref{kinint2}) to obtain
\begin{align}
&|f_2(x+t_{k}-s,s)- f_2(x+t_{k}-s,t_{k-1})|\label{flem2}\\&= \int_{t_{k-1}}^{s} \frac{f_2(x+t_{k}-s,r)}{N_2(r)}
f_1(0,r)dr\le C_bC_{\infty}^2\delta. \nonumber
\end{align}
Substituting (\ref{flem1}) and (\ref{flem2}) into (\ref{int11})
then gives (\ref{int1est}).

By a similar calculation we may use (\ref{ldiff}) to obtain \begin{align}
\mathcal I_2 &\le C_\infty C_b^2 \delta (\Delta L(t_k))^2  \le C_\infty C_b^2
C_2^2\delta^3+\mathcal O(\delta^4). 
\end{align}
Finally, (\ref{tri02}) then comes from collecting estimates for $\mathcal
I = \mathcal I_1+\mathcal I_2$. 
\end{proof}

We are now ready to compare total numbers over length
$\delta$ intervals by defining
\begin{equation}
h_\delta^j(t_k) =  \sup_{l \ge 1} |\mu_j(t_k,I_l)- \tilde  
\mu_j(t_k,I_l)|, \quad h_\delta(t_k) = \sum_{j = 1}^2 h_\delta^j(t_k) . \label{hdef}
\end{equation}
 Our next lemma gives closed recurrence inequalities for $h_\delta(t_k)$.

\begin{lemma} There exists a $C_3(\bar f)$ dependent on initial conditions such that
 for $t_k < T_2(\bar f)$, $h_\delta(t_k)$ satisfies the recurrence inequality
\begin{align}
h_\delta(t_k) \le h_\delta(t_{k-1})+ C_3\left(\delta^2 \sum_{i = 1}^{k-1}
h_\delta(t_i)
+\delta h(t_{k-1})+\omega(\delta,0)\delta^2+\delta^3 \right). \label{hrec}
\end{align}
\end{lemma}
 
\begin{proof} From Lemma \ref{intfest} and (\ref{fdiff}),
for  $l \ge 1,$
\begin{align}
&|\mu_1(t_k,I_l)-\tilde \mu_1(t_k,I_l)|\label{fl1diff}\\
&\le |\mu_1(t_{k-1},I_{l+1})-\tilde
\mu_1(t_{k-1},I_{l+1})|\nonumber\\&+\left|\frac{\Delta L(t_{k})}{N_2(t_{k-1})}\mu_2(t_{k-1},I_{l})-\frac{\Delta\tilde
 L(t_{k})}{\tilde N_2(t_{k-1})}\tilde \mu_2(t_{k-1},I_{l})\right|\nonumber+\mathcal
O(\delta^3+\delta^2\omega(\delta,0)).  \nonumber
\end{align}
Two applications of the triangle inequality yield\begin{flalign}
   &\left|\frac{\Delta L(t_{k})}{N_2(t_{k-1})}\mu_2(t_{k-1},I_{l})-\frac{\Delta\tilde
L(t_{k})}{\tilde N_2(t_{k-1})}\tilde\mu_2(t_{k-1},I_{l})\right| \label{tri1}\\
&\le \left| \left(\frac 1{N_2(t_{k-1})}-\frac 1{\tilde N_2(t_{k-1})}\right)\Delta
L(t_{k})\mu_2(t_{k-1},I_{l})\right|\nonumber\\
&+\left| \frac 1{\tilde N_2(t_{k-1})}\mu_2(t_{k-1},I_l)(\Delta L(t_{k})-\Delta
\tilde L(t_{k}))\right|\nonumber\\&+\left|\frac {\Delta \tilde L(t_{k})}{\tilde
N_2(t_{k-1})}(\mu_2(t_{k-1},I_{l})-\tilde
\mu_2(t_{k-1},I_{l}))\right |. \nonumber
\end{flalign}
We use Lemmas \ref{mbounds} and \ref{nllem}, (\ref{fl1diff}), and  (\ref{tri1})
to obtain 
\begin{align}
 &|\mu_1(t_k,I_l)-\tilde \mu_1(t_k,I_l)| \label{mu1ugly}\\
&\le |\mu_1(t_{k-1},I_{l+1})-\tilde
\mu_1(t_{k-1},I_{l+1})|\nonumber\\&+C_b^2C_2^2\delta ^2|N_2(t_{k-1})-\tilde
N_2(t_{k-1})|
+C_bC_2\delta|\Delta L(t_{k})-\Delta \tilde
L(t_{k})|\nonumber\\&+ C_bC_2\delta|\mu_2(t_{k-1},I_{l})-\tilde
\mu_2(t_{k-1},I_{l})|\nonumber+\mathcal O(\delta^3+\delta^2\omega(\delta,0)).
 \nonumber
\end{align}
Similar bounds hold for Species 2, with \begin{align}
 &|\mu_2(t_k,I_l)-\tilde \mu_2(t_k,I_l)| \label{mu2ugly}\\
&\le |\mu_2(t_{k-1},I_{l})-\tilde
\mu_2(t_{k-1},I_{l})|\nonumber\\&+C_b^2C_2^2\delta ^2|N_2(t_{k-1})-\tilde
N_2(t_{k-1})|
+C_bC_2\delta|\Delta L(t_{k})-\Delta \tilde
L(t_{k})|\nonumber\\&+ C_bC_2\delta|\mu_2(t_{k-1},I_{l})-\tilde
\mu_2(t_{k-1},I_{l})|\nonumber+\mathcal O(\delta^3+\delta^2\omega(\delta,0)).
 \nonumber
\end{align}
We complete the proof by taking the supremum over $l$ for (\ref{mu1ugly})
and  (\ref{mu2ugly}),  using  (\ref{diffncorr}), and then adding to
show
that for some  $C_3$ which depends on $(\bar f_1, \bar f_2)$,
 \begin{align}\label{hrecur}
&h_\delta(t_k) \le h_\delta(t_{k-1})+C_3\left(\delta^2 \sum_{i = 1}^{k-1}
h_\delta(t_i)
 +\delta h_\delta(t_{k-1})
+\omega(\delta,0)\delta^2+\delta^3\right). 
\end{align}
 \end{proof}
We now show that the recurrence inequality (\ref{hrec}) implies  asymptotics
 for $h_\delta(t_k)$.
\begin{lemma} \label{hasympt}
For $t_k < T_2(\bar f)$, 
\begin{equation}
 h_\delta(t_k)= \mathcal O(\delta^2+\delta
\omega(\delta,0)). \label{horder}
\end{equation}
\end{lemma}
\begin{proof}

Let $b_\delta(t_k)$ satisfy the recurrence
equation
\begin{align}
b_\delta(t_k) = b_\delta(t_{k-1})+ C_3\left(\delta^2 \sum_{i = 1}^{k-1} b_\delta(t_i)
+\delta b_\delta(t_{k-1})+\delta \right),\label{brec}
\end{align}
with initial condition $b_\delta(0) = h_\delta(0) = 0.$ It follows immediately from induction
that 
 $h(t_k) \le (\delta^2+ \omega(\delta,0)\delta)) b_\delta(t_k)$ for all $k\ge
0$. Thus, it is sufficient to show that $b_\delta(t_k)
= \mathcal O(1)$ for
all $t_k < T_2(\bar f)$.  To see that this holds, note that as $\delta \rightarrow
0$, (\ref{brec})  converges  to  the linear  integro-differential
equation 
\begin{align}
\tilde b'(t) = C_3\left(\tilde b(t)+\int_0^t\tilde
b(s)ds+1 \right),\quad  t \in [0, T_2(\bar f)), \label{bdiffy}
\end{align}
with initial condition $\tilde b(0) = h_{\delta}(0)$. This can be solved  through elementary
second order methods to obtain
the locally bounded solution\begin{equation}
\tilde b(t) = A_1e^{r_1t}+A_2e^{r_2t},
\end{equation} 
where $A_1= -A_2 =  C_3/(r_2-r_1),$
and $r_1,r_2$ are the two real solutions of $r^2-C_3r-C_3 = 0$. 
\end{proof}

We remark that this theorem also holds when $h_\delta(0) =\mathcal O(\delta^2+\delta
\omega(\delta,0))$.  This is important for extending the time interval of existence. 
\subsection{Proof of Theorem \ref{detscheme} } \label{poofdettheo}

With bounds on differences of total numbers restricted to an interval, we are  finally
ready to show Theorem \ref{detscheme}.  We will require two more lemmas,
which are straightforward to show. First, we  give a formula computing KS
distances when only given information about cumulative functions on grid
points and  bounds on growth between grid points.

\begin{lemma} \label{ksgrid}
For $j = 1,2$, suppose we have measures $\nu_j\in \mathcal M(\mathbb R_+)$
with corresponding cumulative functions $F_j(x) = \nu_j([0,x])$, each with
 compact support $[0,M]$. For $\delta>0$, we can bound the  KS metric by
\begin{align}
&d_{KS}(\nu_1, \nu_2) \le \sum_{i = 1}^{\lceil M/\delta \rceil}|\nu_1(I_i)-\nu_2(I_i)| + \sup_{|x-y|=
\delta} (|F_1(x)-F_1(y)|+|F_2(x)-F_2(y)|). \nonumber
\end{align}
\end{lemma}

We will also require a lemma comparing differences of solutions of kinetic
equations under small changes in time.

\begin{lemma} \label{gridalign}
If $|t_1-t_2| \le \delta$, then 
\begin{equation}
d((\mu_1(t_1), \mu_2(t_1)),(\mu_1(t_2), \mu_2(t_2)) )= \mathcal O( \delta+
\omega(\delta,0)). 
\end{equation}
\end{lemma}

\begin{proof}
Similar to the proof of  Lemma \ref{nllem}.
\end{proof}

\textit{Proof of Theorem \ref{detscheme}}.
Let $t \in [0,T_2(\bar f)]$,  $\delta>0$,  and $K = \lfloor \frac{t}{\delta}\rfloor$, which means $t_K$ is the largest discretized time which is at most $t$.
From Lemmas \ref{mbounds}, \ref{hasympt}, \ref{ksgrid}, and \ref{gridalign},
and recalling that $\tilde \mu_j(t)$ is constant in intervals $t\in [t_{k-1},
t_k)$,  we compute\begin{align}
&d((\mu_1(t), \mu_2(t)), (\tilde \mu_1(t), \tilde \mu_2(t)))\label{distres}\\&\le
d((\mu_1(t), \mu_2(t)), ( \mu_1(t_K),  \mu_2(t_K)))+d((\mu_1(t_K), \mu_2(t_K)),
(\tilde \mu_1(t_K), \tilde \mu_2(t_K)))\nonumber\\&\le \sum_{j = 1}^2\sum_{l
= 1}^{\lceil M/\delta \rceil} |\mu_j(t_K,I_l)-\tilde \mu_j(t_K,I_l)|+
2\left(m_\delta(t_K)+\tilde m_\delta(t_K)\right)+\mathcal O( \delta+ \omega(\delta,0))
\nonumber\\& \le 
2\lceil M/\delta \rceil h_\delta(t_K)+
2(m_\delta(t_K)+\tilde m_\delta(t_K)) +  \mathcal O( \delta+ \omega(\delta,0))\nonumber\\
&=  \mathcal O( \delta+ \omega(\delta,0)). \nonumber
\end{align}

Finally, let us argue for extending the time interval of existence to any $T'<T(\bar f)$.  Consider initial conditions $\mu^{(2)}(0) = \mu(T_2(\bar f))$ and $\tilde \mu^{(2)}(0) = \tilde \mu(T_2(\bar f))$, and new time bounds

\begin{equation}
T_1^{(2)}(\bar f) = \sup\{t: N_2^{(2)}(t)> 2N_2^{(2)}(0)/3\}, \quad T_2^{(2)} = T^{(2)}_1 \wedge 1/(8\tilde C_{\infty}\tilde C_b), 
\end{equation}
with
\begin{align}
\tilde C_{\infty} = \max_{j = 1,2}\sup_{s \in [0,T'(\bar f)]} \|f_j(x,s)\|_{\infty},\quad 
\tilde C_b=2/\inf_{t<T'}N_2(t).  
\end{align}
Then for sufficiently small $\delta$, it follows that  Lemmas \ref{moclem}-\ref{hasympt} hold, with possibly larger constants, and therefore Theorem \ref{detscheme} holds for the time interval $[0,T_{2}+T_2^{(2)}]$ from stitching solutions.  This argument may be repeated to produce $T^{(k)}$.  Each new time interval either adds the constant $ 1/(8\tilde C_{\infty}\tilde C_b)$ (which does not depend on $k$) or reduces $N_2$ by at least 4/5.  After finitely many extensions we will reach a time $T^{(K)}$ at which $N_2(T^{(K)})$ is arbitrarily small, so that $t$ is arbitrarily close to $T(\bar f)$. This completes the proof of Theorem \ref{detscheme}.

\section{Comparison of  particle system and  deterministic scheme} \label{part}

\subsection{Stochastic analogues of  Section \ref{disc}}
We now compare the discretized measures described in the Section \ref{disc} with
the $n$-particle
PDMP $\{X^n(t)\}_{t\ge 0}$. To do so, it will help to express evolution of $\mu_\sigma^n$ from $t_k$ to $t_{k+1}$
in a form that is similar to the iterative formulas (\ref{updatediscneg1})-(\ref{updatedisc0}).
For defining analogous notation to the the discretization scheme, we denote the number of mutations occurring in the time interval $t\in \Delta t_k =   [t_{k-1}, t_k)$  as  $\Delta L^n(t_k)n$.  At each 
  $i \in 1, \dots, \Delta L^n(t_k)n$,   mutation time
 $\tau_i^k$ denote when a particles at position  $x_i\ge 0$ in Species
2 mutates into Species 1. For particles in Species 2 which mutate in  interval $I$ during $\Delta
t_k $, the empirical measure of their positions at  mutation times  is defined by
\begin{equation}
\pi^n_2(t_k,I) = \frac 1n\sum_{i = 1}^{\Delta L^n(t_k)n} \mathbf
1(x_i \in I).
\end{equation}
We define $\pi_1^n$ as the empirical measure for the positions of mutated particles at time $t_k$, with \begin{align}
\pi^n_1(t_k,I)&= \frac 1n\sum_{i = 1}^{\Delta L^n(t_k)n} \mathbf
1(x_i \in I+\tau_i-t_{k-1})\\
&:=\frac 1n\sum_{i = 1}^{\Delta L^n(t_k)n} Q_k^i(I).\nonumber 
\end{align}
Updates for measures on intervals during a time step may then be succinctly written
as\begin{align}
\mu^n_1(t_k,I) &=  \mu^n_1(t_{k-1},I+\delta)+\pi^n_1(t_{k},I) \label{pdmprec1}\\
 \mu^n_2(t_k,I) &=   \mu^n_2(t_{k-1},I)-\pi^n_2(t_{k},I). \label{pdmprec2}
\end{align}
From (\ref{nfcond}), the time interval of existence, under sufficiently many particles $n > n_0(\bar f)$,   before reaching the cemetery state is at least $T_1(\bar f)$, but since we are comparing the particle system with the deterministic discretization, we will work with the smaller time interval $[0, T_2(\bar f)]$.

Let us present variables to be used for the particle system which are similar to those found in Section \ref{disc}.
We begin
with an analogue to $m_\delta$ given in (\ref{mdef}) for 
defining the maximum total numbers  of length $\delta$ intervals  as
\begin{equation}
m^{j;n}_\delta(t) = \sup_{|I| \le \delta} \mu_j^n(t,I),
\quad m_\delta^n(t) =\sum_{j
= 1,2} m^{j;n}_\delta(t). 
\end{equation}From (\ref{pdmprec1})-(\ref{pdmprec2}), it follows for all
realizations of $X^n(t)$ that \begin{equation}
m_\delta^n(t_k) \le m_\delta^n(t_{k-1})+\sum_{j = 1}^2\sup_{I,|I|\le \delta}\pi^n_j(t_{k},I).
\label{mbound}
\end{equation}

There is also a particle system analog of $h_\delta$, where we now compare total number in intervals between $X^n(t)$ and the discretization scheme as 
\begin{equation}
h_\delta^{j;n}(t_k) = \sup_{l \le M/\delta 
 } |\mu_j^n(t_k,I_l)- \tilde
\mu_j(t_k,I_l)|, \quad h_\delta^n=  \sum_{j = 1}^2 h_\delta^{j;n}(t_k). \label{hndef}
\end{equation}
The use of measures $\tilde \mu(t_k)$ rather than $\mu(t_k)$ comes from ability
to use the recurrence (\ref{recmu1})-(\ref{recmu2}), along with
 (\ref{pdmprec1})-(\ref{pdmprec2}) and (\ref{updatediscneg1})-(\ref{updatedisc0}), to write 
 the  recurrence inequality \begin{align}
h^n_\delta(t_k)  &\le h^n_\delta(t_{k-1}) +\sum _{j = 1}^2\max_{ l \le  \tilde
 M/\delta}
\left|\pi_j^n(t_k,I_l)-\frac{\Delta\tilde
 L(t_{k})}{\tilde N_2(t_{k-1})}\tilde \mu_2(t_{k-1},I_l)\right|\label{hmainrec}\\
 &:=h^n_\delta(t_{k-1}) + \Pi^n (t_k). \nn 
\end{align}

From (\ref{mbound}) and (\ref{hmainrec}), the major task for controlling $h^n$ and $m^n$ will clearly hinge on appropriate estimates for $\pi_j$ (and subsequently $ \Pi^n$). These key bounds are provided in the next two lemmas.  
\subsection{Concentration bounds for $\pi_j$}

We begin with an easy to establish generalization of the   Hoeffding
 inequality,  which 
states that for  $n\ge 1$ and 
with $B_i(n,p) \sim \mathrm{Binom}(n, p)$,    
\begin{equation}
\mathbb P\left(\left|\frac{B_i(n,p)}{n}- p\right|>\varepsilon\right) \le
2e^{-2n\varepsilon^2}.
\end{equation}

\begin{cor}\label{hoofthing}  For $n \ge 1$, let   $X^n = \sum_{i = 1}^n
Z_i$, where   $Z_i \sim \mathrm{Ber}(p_i)$ are independent Bernoulli random variables with parameters $0 \le \underline  p \le p_i \le
\bar
p \le 1$. Then for  $\varepsilon>0$, 
\begin{align}
\mathbb P\left( \frac{X^{n}}{ n}-\bar p
> \varepsilon\right)
&\le  2e^{-2n\varepsilon^2}\quad  \hbox{and} \quad   \mathbb P\left( \frac{X^{n}}{
n}-\underline p
<- \varepsilon\right)\le  2e^{-2n\varepsilon^2}\label{hoffgen1} .
\end{align}   

\end{cor}

%
 
%
%
%

For the next two lemmas, we will be interested in cases where parameters
$p_i$ are themselves $[0,1]$-valued random variables with known lower and
upper bounds.
In particular,  we will use  the  \emph{mutation
probability} $P_k^1(t,I)$ defined as the state-dependent probability that if a mutation occurs
at time $t\in [t_{k-1}, t_k)$, then the mutated particle would be located
in $I$ at time $t_k$. We also define $P_k^2(t,I)$ as the 
probability that a particle would mutate in $I$ from Species 2 at time $t$.
  Thus,  
\begin{align}
P_k^1(t,I) =\frac{\mu_2^n(t^-,I+t_k-t)}{N_2^n(t^-)}, \quad 
P_k^2(t,I) =\frac{\mu_2^n(t^-,I)}{N_2^n(t^-)}.
\end{align}
 For an initial  distribution $(\mu^n_1(0), \mu_2^n(0))$ with support
$[0, M]$, tracking which intervals mutations occur in during a time interval  $[t_{k-1}, t_k)$ can be represented
through a random sum  of multinomials  of one draw with random selection probabilities.
In particular, we perform a total of $\Delta L^n(t_k)$ draws with $\lceil
M/\delta \rceil $
bins, in which for each draw the $i$th bin has the mutation probability $P_k^2(\tau_i^-,I)$
of being selected. 

The following two lemmas comparing $\pi_j$ with bounds for $P_k^j(t,I)$ and $\Delta L^n(t_k)$ involve using Corollary \ref{hoofthing} along with the strong Markov property of PDMPs introduced by Davis \cite{dav84}. For the first lemma, we consider
an initial
 state $X^n = X^n(0)$  which has pathwise bounds during $\Delta t_1$ for $\Delta
L^n(t_1), $ and $P^j_1(t,I_l)$. The second lemma assumes these bounds occur with some probability which may be less than 1.
Since $X^n(t)$  is homogeneous, both these lemmas are readily applicable when considering transitions during  $\Delta t_k$ for $k>1$.

We will use common notation for stochastic
ordering: for real-valued random variables
$X$ and $Y$, we write $X  \le_{ST}Y$ if $\mathbb P(X>c) \le  \mathbb P(Y>c)$
for all real $c$.  Also note  in the next two lemmas, we will suppress
time arguments when no confusion
will arise. Finally, for simplicity with presentation, we will assume that $M/\delta$  is integral.

\begin{lemma} \label{lohidevroye}  Let $X^n(0) =  \mathbf x \in E$ be an initial state such 
that for $ l \in 1,\dots, M/\delta$, $ j \in \{1,2\}$, and $t\in \Delta t_1$, 
\begin{align}
 \underline L\le  \Delta L^n \le \bar L \quad \hbox{ and }  \label{devbnd1} \quad  \underline p_l\le P_1^j(t,I_l) \le \bar p_l, 
\end{align}
  where $\underline L, \bar L, \underline p_l, \bar p_l$ are $[0,1]$-valued constants. Then for $\varepsilon > 0$,  
\begin{align}
\mathbb P\left(\max _{l\le M/\delta} \left(\pi^n_j(t_1,I_l)-\bar
L\bar p_l\right)
> \varepsilon \right) \le\frac {2M}{\delta}\mathbb
\exp(-2n\varepsilon^2/\bar L) \label{mainpithm1}
\end{align}
and
\begin{align}
\mathbb P\left(\min _{l\le M/\delta}\left(\pi^n_j(t_1,I_l)-\underline
L\underline
p_l
\right)<- \varepsilon \right) \le\frac {2M}{\delta} \exp(-2n\varepsilon^2/\underline L).
\end{align}

 \end{lemma}
 
 \begin{proof}
  
We will show (\ref{mainpithm1}) for $j = 1$.  The proofs
for the other cases are similar.   For the parameter $q
\in [0,1]$, denote $\{B_i(q)\}_{i \ge 1}$ as an iid stream of Bernoulli random
variables with $B_1 \sim \mathrm{Ber}(q)$. We write
\begin{equation}
\bar Q_{l}^i := B_i(\bar
p^n_l), \quad  Q_{l}^i :=\begin{cases}Q^i_1(I_l) & i \le\Delta
L^nn, \\
\bar Q_{l}^i & i >\Delta L^nn  \\
\end{cases}.
\end{equation} 
We use iterated conditioning
to show
that \begin{align}
\mathbb P(Q_{l}^i = 1) = \mathbb E[Q_{l}^i ] =\mathbb E[\mathbb E[Q_{l}^i |X(\tau_1^{i-})]   
 ]= \mathbb E[P_1^1(\tau_1^{i-},I_l)] \label{argstart}  \le \bar p_l =\mathbb P(\bar Q_{l}^i=
1). 
\end{align}
 This calculation implies the stochastic ordering $Q_{l}^i \le_{ST} \bar Q_{l}^i$.  
 
Next, we show
\begin{equation}
\pi^n_1(I_l)\le_{ST} \frac 1n\sum_{i = 1}^{\bar Ln}
 Q_{l}^i\le_{ST} \frac 1n\sum_{i = 1}^{\bar Ln}
\bar Q_{l}^i .   \label{qineqs}
\end{equation}
The left inequality is immediate, and in fact holds for all paths in $X^n(t)$.
To show the right inequality, we use induction, assuming
that for   $1\le j <\bar Ln$,
\begin{align}
\mathbb P\left(\sum_{i = 1}^{j} Q_{l}^i>c\right) &\le \mathbb
P\left(\sum_{i = 1}^{j} \bar Q_{l}^i >c\right). 
\end{align}

The base case holds trivially. For the inductive step, we use the low of total probability to show
\begin{align}
\mathbb P\left(\sum_{i = 1}^{j+1} Q_{l}>c\right) &= \mathbb
E\left[\mathbb P\left(\sum_{i = 1}^{j} Q_{l}^i
+Q_{l}^{j+1} >c \Big | X((\tau_1^{j+1})^-)\right)\right]\label{argend}\\
&\le  \mathbb E\left[\mathbb
P\left(\sum_{i = 1}^{j}  Q_{l}^i +\bar Q^{j+1}_{l}>c
\Big |  X((\tau_1^{j+1})^-)\right)\right]\nn \\
 &= \mathbb P\left(\sum_{i = 1}^{j}  Q_{l}^i +\bar Q^{j+1}_{l}>c
\right)\le     \mathbb P\left(\sum_{i = 1}^{j+1} \bar Q_{l}^i
>c\right). \nn 
\end{align}
The first inequality in (\ref{argend}) uses a well-known property for stochastic
dominance when summing random variables: if $X_1$ and $X_2$ are independent,
$Y_1$ and  $Y_2$ are independent, and $X_i \le_{ST} Y_i$ for $i = 1,2$, then
$X_1+X_2 \le_{ST} Y_1+Y_2$. From the strong Markov property of PDMPs, the $\mathcal F(\tau_k^{j})$-measurable quantity   $\sum_{i = 1}^{j} Q_{k;l}^i  $ and the $\mathcal F(\tau_k^{j+1})$-measurable quantity    $ Q_{k;l}^{j+1}
$ are conditionally independent
under  $\mathbb P(\cdot|X((\tau_k^{j+1})^-)$. The last
inequality uses the same property of stochastic dominance along with
the induction hypothesis.   

With (\ref{argend}) we  then obtain our result 
from Lemma \ref{hoofthing}, with\begin{align}
 & \mathbb P\left(\max _{l\le M/\delta} \left( \pi^n_1(I_l)
 -\bar L\bar p_l\right) 
> \varepsilon  \right) 
\le \sum_{l\le M/\delta}\mathbb P\left( \pi^n_1(I_l)
 -\bar L\bar p_l > \varepsilon\right) \label{condo1}\\
 &\le \sum_{l\le M/\delta} \mathbb
P\left(  \frac 1n\sum_{i = 1}^{\bar Ln} \bar Q_{l}^i- \bar L\bar p_l
> \varepsilon\right)
\le\frac {2 M}{\delta} \exp(-2n\varepsilon^2/\bar L). \nn
\end{align}
\end{proof}

 While Lemma \ref{lohidevroye} is useful for  computing total numbers in Species 1 and 2, we find in Section \ref{hdiffsec} that for comparing the PDMP to our deterministic discretization, it is necessary to consider another estimate of $\pi_i$ in which the bounds on $\Delta L(t_k)$
and $P_k^j$ may not hold with small probability. This differs from Lemma \ref{lohidevroye}, in which assume such bounds occur over all paths given an appropriate initial condition.

For the next lemma and in many other places, we will frequently rely on an elementary inequality derived from the law of total probability:\ for events $C$
and $D$,
\begin{equation}
\mathbb P(C) = \mathbb P(C|D)\mathbb P(D) + \mathbb P(C|D^c)\mathbb P(D^c)\le
 \mathbb P(C|D)+ \mathbb P(D^c). \label{probid}
\end{equation}

\begin{lemma} \label{lohidevroye2} Let  $\mathcal T
$  be an event such that 
\begin{align}
&\mathcal T \subset \{  \underline   L\le  \Delta L^n\le \bar L\} \cap \mathcal \{P_1^j(t,I_l) \in[
\underline p_l, \bar p_l]:l
\le M/\delta, j = 1,2, t \in \Delta t_1 \},
\nn
\end{align}  
where      $\underline L, \bar
L,\underline p_l,$ and  $\bar p_l$ are  $[0,1]$-valued constants.  Suppose for some $ r(\delta,
n)\in [0,1)$ that  
\begin{equation}
\mathbb P(\mathcal T^c) \le r(\delta,n).  \label{tbound}   
\end{equation}Then for $j = 1,2$ and $\varepsilon >0$,
\begin{align}
&\mathbb P\left(\max _{l\le M/\delta} \left(\pi^n_j(t_1, I_l)-\bar
L\bar
p_l\right)
>  \varepsilon\Big|\mathcal T\right) \le\frac {2M}{\delta}\mathbb
\exp(-2n\varepsilon^2/\bar L)+\frac M \delta\left( \frac{r}{1-r}+\frac{2r\bar
Ln}{(1-r)^2}\right)
\label{mainpithm4}
\end{align}
and\begin{align}
&\mathbb P\left(\min _{l\le M/\delta}\left(\pi^n_j(t_1, I_l)-\underline
  L\underline
p_l
\right)<- \varepsilon \Big|\mathcal T \right)\le\frac {2M}{\delta}\mathbb
\exp(-2n\varepsilon^2/\underline L) +\frac M \delta( 2r\bar
Ln+r). \label{mainpithm5}
\end{align}
 \end{lemma}

\begin{proof}
We show (\ref{mainpithm4}) and (\ref{mainpithm5}) for $j = 1$. 
We first show by induction that 
\begin{align}
\mathbb P\left(\sum_{i = 1}^{j} Q_{l}>c \right)\le\mathbb P\left(\sum_{i = 1}^{j} \bar Q_{l}>c \right)+ \frac{2rj}{1-r}, \quad j= 1, \dots, \bar L n. \label{qind}
\end{align}
We will condition on the $\mathcal
F(X(\tau_1^{i-}))$-measurable event\begin{equation}
\mathcal T_i= \{\underline p_l\le  P_1^1(\tau_1^{i-},I_l)
\le \bar p_l\} \supseteq  \mathcal T.
\end{equation}

The base case for (\ref{qind}) follows from (\ref{probid}) and (\ref{tbound}) since
\begin{align}
\mathbb P\left( Q_{1}>c \right) &\le \mathbb P\left( Q_{1}>c|\mathcal T_1 \right)+\mathbb P(\mathcal
T_{1} ^c)\le \frac{\mathbb P\left( \bar Q_{1}>c
\right)}{1-r}+r\\ 
&\le  \mathbb P\left( \bar Q_{1}>c
\right)+\frac{r}{1-r}+r\le \mathbb P\left( \bar Q_{1}>c
\right)+\frac{2r}{1-r}. \nn
\end{align}
For the inductive step, assuming (\ref{qind}) holds for $ 1\le j<\bar Ln $,  we use the strong Markov property of PDMPs to show 
\begin{align}
&\mathbb P\left(\sum_{i = 1}^{j+1} Q_{l}>c \right) \le \mathbb P\left(\sum_{i = 1}^{j} Q_{l}+Q_{j+1}>c\Big |\mathcal T_{j+1} \right)+\mathbb P(\mathcal T_{j+1} ^c)&\\&\le \mathbb P\left(\sum_{i
= 1}^{j} Q_{l}+\bar Q_{j+1}>c\Big |\mathcal T_{j+1} \right)+r \le \mathbb P\left(\sum_{i
= 1}^{j} Q_{l}+\bar Q_{j+1}>c \right)+\frac{2r}{1-r} \nn \\
&\le  \mathbb
P\left(\sum_{i
= 1}^{j+1} \bar Q_i>c \right)+\frac{2rj}{1-r} + \frac{2r}{1-r}  = \mathbb
P\left(\sum_{i
= 1}^{j+1} \bar Q_{i}>c \right)+\frac{2r(j+1)}{1-r}. \nn
\end{align}

From calculations similar to (\ref{argstart})-(\ref{argend}), we use (\ref{qineqs}) to show 
\begin{align}
&\mathbb P(\pi^n_1(I_l)>c |\mathcal T) \le \mathbb P\left(\frac 1n\sum_{i
= 1}^{\bar Ln} Q_{l}>c\Big |\mathcal T\right)\\&\le \frac 1{1-r}\mathbb
P\left(\frac 1n\sum_{i = 1}^{\bar Ln} Q_{l}>c \right) \le\frac 1{1-r}\mathbb
P\left(\frac 1n\sum_{i = 1}^{\bar Ln} \bar Q_{l}>c \right)+ \frac{2r\bar Ln}{(1-r)^2} \nn\\
&\le \mathbb
P\left(\frac 1n\sum_{i = 1}^{\bar Ln} \bar Q_{l}>c \right)+ \frac{r}{1-r}+\frac{2r\bar
Ln}{(1-r)^2}. \nn
\end{align}

For the lower bound, we note 
\begin{align}
\mathbb P(Q_{l}^i = 1)\ge \mathbb P(Q_{l}^i= 1|\mathcal T_i)\mathbb
P(\mathcal T_i) \ge   \mathbb P(\underline Q_{l}^i = 1)-r(\delta, n). \label{baseq}
\end{align}
We also note, for an event $A$ and $j \le \bar Ln$, \begin{align}
\mathbb P(A)-r \le \mathbb P(A)-\mathbb P(\mathcal T_{j} ^c) \le  \mathbb P(A)-\mathbb P(\mathcal T_{j} ^c)\mathbb P(A|\mathcal T_{j} ^c) =\mathbb P(\mathcal T_{j} )\mathbb P(A|\mathcal T_{j} ) \le  \mathbb P(A|\mathcal T_{j} ).  
\end{align}
We again use induction to show 
\begin{equation}
\mathbb P\left(\sum_{i = 1}^{j} Q_{l}>c \right) \ge \mathbb
P\left(\sum_{i
= 1}^{j} \bar Q_{i}>c \right)-2rj,  \quad j= 1, \dots, \bar L
n. \label{qind2}
\end{equation}
 Showing the base case is similar to (\ref{baseq}). For the induction step,
 \begin{align}
&\mathbb P\left(\sum_{i = 1}^{j+1} Q_{l}>c \right) \ge  \mathbb P\left(\sum_{i
= 1}^{j+1} Q_{l}>c\Big |\mathcal T_{j+1} \right)\mathbb P(\mathcal
T_{j+1} ) \ge   \mathbb P\left(\sum_{i
= 1}^{j} Q_{l}+Q_{j+1}>c\Big |\mathcal T_{j+1} \right)-r\\&\ge \mathbb P\left(\sum_{i
= 1}^{j} Q_{l}+\underline Q_{j+1}>c\Big |\mathcal T_{j+1} \right) -r\ge  \mathbb
P\left(\sum_{i
= 1}^{j} Q_{l}+\underline Q_{j+1}>c \right)-2r \nn\\
&\ge  \mathbb
P\left(\sum_{i
= 1}^{j+1} \underline Q_i>c \right)-2r(j+1). \nn
\end{align}
With (\ref{qind2}), we then can show
\begin{align}
&\mathbb P(\pi^n_1(I_l)>c |\mathcal T) \ge  \mathbb P\left(\frac 1n\sum_{i
= 1}^{\underline Ln} Q_{l}>c\Big |\mathcal T\right)\\&\ge \mathbb
P\left(\frac 1n\sum_{i = 1}^{\underline Ln} Q_{l}>c \right) -r\ge\mathbb
P\left(\frac 1n\sum_{i = 1}^{\underline Ln} \underline Q_{l}>c \right)- 2r\bar
Ln-r. \nn 
\end{align}
Upon taking complements,
we arrive at\begin{equation}
\mathbb P(\pi^n_1(I_l)<c |\mathcal T) \le \mathbb P\left(\frac 1n\sum_{i = 1}^{\underline Ln} \underline Q_{l}<c \right)+ 2r\bar
Ln+r. \label{pigreat}
\end{equation}

The result then follows from mirroring  the calculations of (\ref{condo1}), with 
 \begin{align}
& \mathbb P\left(\max _{l\le M/\delta} \left( \pi^n_1(I_l)
 -\bar L\bar p_l\right) \nn
> \varepsilon \Big|\mathcal T  \right) \le \sum_{l\le M/\delta}\mathbb P\left( \pi^n_1(I_l)
 -\bar L\bar p_l
> \varepsilon \Big|\mathcal T\right)
\\&\le \sum_{l\le M/\delta}
\mathbb P\left(\frac 1n \sum_{i = 1}^{\bar Ln} \bar Q_{l}^i-\bar
L\bar p_l
> \varepsilon\right)+\frac M \delta\left( \frac{r}{1-r}+\frac{2r\bar
Ln}{(1-r)^2}\right) \nn
\\
&\le \frac {2M}{\delta}\mathbb
\exp (-2 n\varepsilon^2/\bar L) +\frac M \delta\left( \frac{r}{1-r}+\frac{2r\bar
Ln}{(1-r)^2}\right). \label{condo6}
\end{align}
A similar calculation using (\ref{pigreat}) yields (\ref{mainpithm5}).
\end{proof}

\subsection{PDMP lemmas on growth}

In this section, we derive bounds for total number in an interval. Our first estimate is a simple pathwise bound between time intervals. 
\begin{lemma}\label{crudelem} For all realizations of $X^n(t)$, if $s
\in [t_{k-1}, t_k)$, then
\begin{equation}
m_\delta^n(s) \le 3m_\delta^n(t_{k-1}). \label{firstmbound}
\end{equation}
\end{lemma}
\begin{proof}
For Species 2, note that $m^{2;n}_\delta(t)$ is decreasing in $t$.  As for
Species 1, a particle in an interval $I$ of size $\delta$ at time $s\in
[ t_{k-1}, t_k)$
must have been located, at time $t_{k-1}$, in $I$ in Species 2 or  $I+\delta$
in either  Species 1 or 2, so that
\begin{align}
\mu_1^n(s, I)&\le\mu_1^n(t_{k-1}, I)+\mu_1^n(t_{k-1}, I+\delta )+\mu_2^n(t_{k-1},
I+\delta),\\
\mu_2^n(s, I) &\le\mu_2^n(t_{k-1},
I).
\end{align}
Taking the supremum over length $\delta$ intervals then  yields (\ref{firstmbound}).
  \end{proof}

From (\ref{nfcond}), there is $n_1(\bar f) \ge n_0(\bar f)$ such that for
  $n>n_1(\bar
f)$,  we can use the constant 
\begin{equation}
C_4 = \inf_{n \ge n_1} N^n_2(T_2(\bar f))>N_2(0)/2.
\end{equation} 
For our next lemma, we compare $m_\delta^n(t_k)$ with its deterministic analog $\bar m_\delta(\tau)$, defined through the    recurrence 
\begin{equation} \label{mrecur}
\bar m_\delta(t_k) = \bar m_\delta(t_{k-1})+ 24C_4\bar m_\delta(t_{k-1})^2,
\quad
\bar m_\delta(0) = m_\delta^{n}(0). 
\end{equation}
We may use the same reasoning as in Lemma \ref{mbounds}   to show that there exists $\delta^p>0$ such that for $0 < \delta <\delta^p$, we can find   $n^{p}(\delta,\bar f)>n_1(\bar f) $ such that for $n>n^{p}(\delta,\bar
f)$ there exist positive constants $C_5,\hat C_5>0$ such
that 
\begin{equation}\label{mbarbound}
\hat C_5 \delta\le \bar m_\delta(t_k) \le C_5\delta.
\end{equation}
For the remaining lemmas in this section, we will assume that  $0 < \delta <\delta^p$ and $n> n^p(\delta)$.

Our interest is in whether interval growth in the particle system exceeds that of $\bar m_\delta$.  Whether this occurs is  expressed in the sequence of events 
\begin{equation}
\mathcal A_k = \{m_\delta^n(t_k) >  \bar m_\delta(t_k)\}, \quad 0 \le t_k \le T_2.
\end{equation}
Our next lemma shows that conditioned under $\mathcal A_{k-1}^c$, we can
use Lemma \ref{lohidevroye} to obtain a  concentration inequality for $\pi_j$ and consequently $m_\delta^n(t_k)$.

\begin{lemma}\label{pilem}
Let $j = 1,2.$ For     $0<  t_k \le T_{2}$, 

\begin{equation}\label{pibound}
\mathbb P\left(\sup_{|I|\le \delta}\pi^n_j(t_k,I)> 12C_4  
\bar m_\delta(t_{k-1})^2|\mathcal
A_{k-1}^c\right) \le \frac{2M}{\delta}\exp(-\tilde C_6\delta^3 n),\end{equation}
where   $ \tilde C_6 = 72C_4^2 \hat C_5^3$.
\end{lemma}

\begin{proof}
  We give a proof for $j = 1$, with the proof for $j = 2$ being nearly
identical. First, observe that under $\mathcal 
A_{k-1}^c$ there will be
at most $\Delta L^n(t_k) n\le 
\bar m_\delta(t_{k-1})n$ mutations during $\Delta t_k = [t_{k-1},t_k)$, , since  only particles contained  in
$[0,\delta)$ at $t_{k-1}$ for Species 1 and 2 may possibly reach the origin in Species
1. Under the event  $\mathcal 
A_{k-1}^c$, we use Lemma \ref{crudelem} to show mutation probabilities are uniformly  bounded from above by the
deterministic quantity
\begin{equation}\label{onemutation}
P_k^1(\tau,I)  = \frac{\mu_2^n(t^-,I+t_k-t)}{N_2^n(t^-)}\le 3C_4 \bar m_\delta(t_{k-1}).
\end{equation}
 Thus, using From (\ref{mbarbound}), we can then use Lemma \ref{lohidevroye}
with $\bar L =\bar m_\delta(t_{k-1})$,  
$\bar p_l =    3C_4 
 \bar m_\delta(t_{k-1})$, and $\varepsilon =6C_4 (\bar m_\delta(t_{k-1}))^2$, which gives 
\begin{align}\label{intpibound}
\mathbb P\left(\max_{l\le M/\delta}\pi^n_1(t_k, I_l)>6C_4 (\bar m_\delta(t_{k-1}))^2|\mathcal
A_{k-1}^c\right) \le\frac{2M}{\delta}\exp(-\tilde C_6\delta^3 n). 
\end{align}
Note that to apply Lemma \ref{lohidevroye}
we used the fact that $X^n(t)$ is homogeneous.
Indeed,
we can write the left hand side of (\ref{intpibound}) as
\begin{equation}
\mathbb P\left(\max_{l\le M/\delta}\pi^n_1(t_k, I_l)>6C_4 (\bar m_\delta(t_{k-1}))^2|\mathcal
A_{k-1}^c\right) = \mathbb P\left(\max_{l\le M/\delta}\pi^n_1(t_1, I_l)>6C_4 (\bar m_\delta(t_{k-1}))^2|\tilde{\mathcal A}^c\right),
\end{equation}
where $\tilde {\mathcal A} =  \{m_\delta^n(0) >  \bar m_\delta(t_{k-1})\}$ gives requirements on the initial condition so that $\Delta L(t_1)\le \bar L$, and $P_k^1(\tau,I_l) \le \bar p_l$ pathwise in $\Delta t_1$. 

We can extend (\ref{intpibound}) to hold over all intervals of size less
than $\delta$, not just those on a grid.  This is done by noting  that for
any measure $\mu$, if $\mu(I_k) \le a$ on
a uniform grid $I_k$ of size $\delta$, then for any $I$ with $|I| \le \delta$,
$\mu(I) \le 2a$. 
Thus, (\ref{pibound}) follows from (\ref{intpibound}) and (\ref{mbarbound}),
since
\begin{align}
&\mathbb P\left(\sup_{I,|I|\le \delta}\pi^n_1(t_k, I)>12C_4  
\bar m_\delta(t_{k-1})^2|\mathcal
A_{k-1}^c\right)\le\mathbb
P\left(\max_{l\le M/\delta}\pi^n_1(t_k, I_l)>6C_4  
\bar m_\delta(t_{k-1})^2|\mathcal
A_{k-1}^c\right). \nn
\end{align}
\end{proof} 
 
We can now derive a concentration inequality which shows  that   $m^n_\delta (t_k) =  \mathcal O(\delta) $ with high probability. 

\begin{lemma}\label{mlem} For $0 \le   t_k \le T_{2},$
\begin{align}\label{probintbnd}
\mathbb P(m_\delta^n(t_k)>C_5\delta ) \le\mathbb P\left(\mathcal A_k\right)\le   \frac{4 Mk}{\delta}\exp(-\tilde
C_6\delta ^3n).
\end{align}

\end{lemma}

\begin{proof}

The proof follows from  induction, in which we assume 
\begin{equation}
\mathbb P\left(\mathcal A_l\right)\le \frac{4Ml}{\delta}\exp(-\tilde
C_6\delta ^3n) \label{abound}
\end{equation}
holds for $0\le l<k$. The base case for  when  $k = 0$ follows since $\mathbb P(\mathcal A_0 ) = 0$. To show  the inductive step,  we can 
use  Lemma \ref{pilem} and the recurrence inequalities (\ref{mbound}) and (\ref{mrecur}) to show
\begin{align}
&\mathbb P(\mathcal A_{k}|\mathcal A_{k-1}^c)  \label{aifnota}\\
& \le \mathbb P\left(m_\delta^n(t_{k-1})+ \sum_{j = 1}^2\sup_{I,|I|\le \delta}\pi^n_j(t_{k},I)>  \bar m_\delta(t_{k-1})+ 24C_4\bar m_\delta(t_{k-1})^2\Big|\mathcal A_{k-1}^c\right) \nn \\&\le \sum_{j = 1}^2\mathbb P\left(\sup_{I,|I|\le \delta}\pi^n_j(t_{k},I)
> 12C_4\bar m_\delta(t_{k-1})^2|\mathcal
A_{k-1}^c\right)\le\frac{4M}{\delta}\exp(-\tilde C_6\delta^3 n).  \nn   
\end{align}

We may then apply (\ref{probid}) with $C = \mathcal A_{k}$ and $D = \mathcal
A_{k-1}^c$,
and then apply (\ref{abound}) for $l = k-1$, along with (\ref{aifnota}) and
(\ref{mbarbound}) to obtain the right inequality of (\ref{probintbnd}) (the left inequality follows immediately from (\ref{mbarbound})).
\end{proof}

We finish this subsection with an estimate for total  mutations during $\Delta t_k$ analogous to Lemma \ref{nllem}.

\begin{lemma}
 \label{Lvscdf}For $0< t_k \le T_{2}$, 

\begin{equation}
\mathbb P(|\Delta L^n(t_k)-\mu_1^n(t_{k-1},[0,\delta])|>C_6\delta^2) \le
\frac{4Mk}{\delta}\exp(-\tilde C_7\delta ^3n), \label{llemma}
\end{equation}
with $C_6  =  4C_2C_5^2$ and $\tilde C_7 = \tilde C_6 \wedge (18C_4 ^2C_5^3)$.

\end{lemma}

\begin{proof}
  
Denote $\mathcal M_{\delta}^n(t_k)$ as the normalized total number of mutations
that
affect Species 2 particles in the  interval $[0,\delta)$ during time $[t_{k-1},t_k)$. This may be written as

\begin{equation}
\mathcal M_{\delta}^n(t_k) = \frac 1n\sum_{i = 1}^{\Delta L(t_k)} M_i^k, 
\end{equation}
where $M_i^k$ is an indicator random variable for the event that the $i$th
mutation during $\Delta t_k$ occurs within $[0, \delta)$. 

All particles of Species 1 located in $[0,\delta)$ at $t_{k-1}$ with reach the origin and trigger a mutation by time $t_k$.  The only other particles in the system which potentially hit the origin are those initially located in $[0,\delta)$ in Species 2 which have mutated during    $\Delta t_k$.  It follows
that
under all paths in $X^n(t)$, \begin{equation}
|\Delta L^n(t_k)-\mu_1^n(t_{k-1},[0,\delta])|\le \mathcal M^n_{\delta}(t_k). \label{mainmut}
\end{equation}
Thus proving (\ref{llemma}) follows from showing an equivalent estimate  on $\mathcal M_{\delta}^n(t_k)$. 
Toward that end, note that  under $\mathcal A_{k-1}^c$
 the number of mutations during time
$\Delta t_k$ is less than  $\bar L n= C_5\delta n $. Also, the probability  of selecting a Species 2 particle to mutate in $[0,\delta]$ at each mutation time during   $\Delta t_k$ is bounded by $\bar p = 3C_4 C_5\delta$. 
Let $B_i(q)$ denote an iid stream of Bernoulli random variables with parameter
$q \in [0,1]$.  It follows that under  $\mathbb P(\cdot|\mathcal A_{k-1}^c)$,
from arguments similar to Lemma \ref{lohidevroye}, that\begin{equation}
\mathcal M_{\delta}^n(t_k) \le_{ST} \frac 1n\sum_{i = 1}^{\bar Ln} B_i(\bar
p). \quad    
\end{equation}

For $C_6 =  6C_2C_5^2$, from the Hoeffding inequality,
\begin{align}\label{mutbound}
\mathbb P(\mathcal M^n_{\delta}(t_k)>C_6\delta^2|\mathcal
A_{k-1}^c)
\le \mathbb P\left(\sum_{i = 1}^{\bar Ln} B_i(\bar p)/(\bar L n)-\bar p>\bar
p\right) &\le 2\exp(-2\bar L\bar p^2 n)  \\
&= 2\exp(-18C_4^2C_5^3\delta^3 n). \nn
\end{align}
The lemma then follows from applying Lemma \ref{mlem}, (\ref{probid}),   (\ref{abound}),and
(\ref{mutbound}).
\end{proof}

\subsection{Difference of total number  on an interval   } \label{hdiffsec}

 We now begin our estimates comparing the deterministic discretization $\tilde \mu_j$ and empirical measures $\mu^n_j$. As in Section \ref{convest}, our focus is on differences of measures restricted to length $\delta$ intervals.

For inequalities related to bounding $P^j_k$ with  $h_\delta^{j;n}$, we will need to consider a modulus of continuity for the
deterministic discretization,
defined by
\begin{equation}
\tilde \omega(s ,t_k) = \sup_{I:|I| =  \delta} \sum_{j = 1}^2 |\tilde
\mu_j(t_k,
I+s)- \tilde \mu_j(t_k,
I)|, 
\end{equation}

Fortunately, we can compare $\tilde \omega(\delta, t_k)$
with $\omega(\delta, t_k)$, the modulus of continuity for the solutions to
the kinetic equations, through the following lemma. 
  
\begin{lemma} \label{omegalem2}
There exists $C_8>0$ such that for $0<s\le \delta$, 
\begin{equation}
\tilde \omega(s, t_k)  \le C_8(\omega(\delta, 0) \delta+\delta^2).
\end{equation}
\end{lemma}

\begin{proof} For $j = 1,2$, 
\begin{align}&|\tilde \mu_j(t_k,
I+s)- \tilde \mu_j(t_k,
I)|\\&\le | \tilde \mu_j(t_k,
I+s)- \mu_j(t_k,
I+s)|+| \mu_j(t_k,
I+s)-  \mu_j(t_k,
I)|+| \mu_j(t_k,
I)- \tilde \mu_j(t_k,
I)|.  \nn
\end{align}

Summing over $j$ and taking suprema gives\begin{align}
\tilde \omega(s, t_k)  \le \omega(s, t_k)\delta  +2h_{ \delta}(t_k).
\end{align}
The result then follows from Lemmas \ref{moclem} and  \ref{hasympt}.
\end{proof}

 We now give a pathwise inequality over $\Delta t_k$  for comparing mutation
probabilities. 
\begin{lemma} \label{pnattk} For $\tau \in \Delta t_k$,
  there exists $C_9>0$ such that for all paths in $X^n(t),$ 
  \begin{align}
|P^j_k(\tau,I)- P^j_k(t_{k-1},I)| \le C_9(h^{n}(t_{k-1})+\pi_2^n(t_k)+(m^n(t_{k-1}))^2+\omega(\delta, 0) \delta+\delta^2).\label{bigplem}
\end{align}
\end{lemma}

\begin{proof}
 For $j = 1$ (the proof for $j = 2$ is similar), we write 
 
 \begin{align}
&|P^1_k(\tau,I)- P^1_k(t_{k-1},I)| = \left| \frac{ \mu_2^n(\tau, I+t_k-\tau)}{N^n(\tau)}+\frac{
\mu_2^n(t_{k-1}, I+\delta)}{N^n(t_{k-1})}\right| \\
 &\le C_4(|\mu_2^n(\tau, I+t_k-\tau)-\mu_2^n(t_{k-1}, I+\delta)|+\mu_2^n(t_{k-1},
I+\delta)(N^n(\tau)-N^n(t_{k-1}))) \nn\\
&\le C_4(|\mu_2^n(\tau, I+t_k-\tau)-\mu_2^n(t_{k-1}, I+\delta)|+(m^n(t_{k-1}))^2.
\nn
\end{align}
We may then break up terms further, with\begin{align}
&|\mu_2^n(\tau, I+t_k-\tau)-\mu_2^n(t_{k-1}, I+\delta)|\\ 
&\le |\mu_2^n(\tau, I+t_k-\tau)- \mu_2^n(t_{k-1}, I+t_k-\tau)|\nn\\&+| \mu_2^n(t_{k-1},
I+t_k-\tau)- \tilde \mu_2(t_{k-1}, I+t_k-\tau)|\nn \\
&+|\tilde  \mu_2(t_{k-1}, I+t_k-\tau)- \tilde \mu_2(t_{k-1}, I+\delta
)|+ |\tilde \mu_2(t_{k-1}, I+\delta)-\mu_2^n(t_{k-1}, I+\delta)|\nn\\
&\le \pi_2^n(t_k)+ 2h^n(t_{k-1})+ \tilde \omega(\delta-t_k-\tau, t_{k-1}).\nn
\end{align}
The result follows from applying Lemma (\ref{omegalem2}).
\end{proof}

 Here we collect previous results to form a high-probability
event under which we can bound $h^n_\delta(t_k)$.
\begin{lemma}
 There exist positive constants  $C_{10}, C_{11}, C_{12} $ such that the events
\begin{align}
\mathcal C(t_k) &= \left\{\sup_l \pi_2^n(t_k,I_l) > C_{10}\delta^2\right\},\\
 \quad \mathcal L(t_k) &= \cup _{t_k \le T_{2}}\{|\Delta L^n(t_k)-\mu_1^n(t_{k-1},[0,\delta])|>C_{10}\delta^2\},
   \\\mathcal D(t_{k+1})&=  \mathcal C(t_{k+1})^c\cap
\mathcal A(t_k) \nn^c\cap \mathcal L(t_{k+1})^c.     
\end{align}
satisfy
\begin{align}
 \mathbb P(\mathcal D(t_k)^c|\mathcal A(t_{k-1})^c) 
 \le \frac{C_{12}}{\delta} \exp(-C_{11}\delta ^3n). \label{newrbound}
\end{align}

\end{lemma}

\proof  This follows immediately from applying  (\ref{pibound}) with  (\ref{mainmut})
and (\ref{mutbound}) to the event\be
 \mathbb P(\mathcal D(t_k)^c|\mathcal A(t_{k-1})^c) 
 \le \mathbb P(  \mathcal C(t_k)|\mathcal A(t_{k-1})^c)+ \sum_{k\le \lceil T_{2}/\delta\rceil } \mathbb P (\mathcal L(t_{k})|\mathcal A(t_{k-1})^c). \nn \qed
\ee

The upshot of using   $\mathcal D(t_k)$ is that we may bound selection probability and total losses by quantities which are $\mathcal F(t_{k-1})$ measurable. Furthermore,  we may  we use $\mathcal D(t_k)$ for the event  $\mathcal T$ in  Lemma \ref{lohidevroye2} to produce concentration bounds for $h(t_k)$. 

\begin{lemma} \label{pandlbounds}
Under $\mathcal D(t_k)$, there exists $C_{13}$ such that
for $t\in \Delta t_k$, \begin{equation}
P_k^j(t,I_l) \in [\underline p_l(t_{k-1}),\bar p_l(t_{k-1})], \qquad   \Delta L(t_k) \in [ \underline L^n(t_{k-1}), \bar L^n(t_{k-1})], 
\end{equation}
where
\begin{align}
\bar p_l(t_{k-1})&=  P^2_k(t_{k-1},I_l)+C_{13}(h^n(t_{k-1})+\omega(\delta,
0) \delta
+\delta^2), \label{pstart}\\
\underline p_l(t_{k-1})&=  P^2_k(t_{k-1},I_l)-C_{13}(h^n(t_{k-1})+\omega(\delta,
0) \delta
+\delta^2), \label{pend}\\
 \bar L^n(t_{k-1}) &=\mu_1^n(t_{k-1},[0,\delta])+C_{10}\delta^2, \label{lstart}\\
  \underline L^n(t_{k-1}) &=\mu_1^n(t_{k-1},[0,\delta]).\label{lend}
\end{align}
\end{lemma} 

\begin{proof}
 We obtain (\ref{pstart})-(\ref{pend})  from  Lemma
\ref{pnattk}, and (\ref{lstart})-(\ref{lend}) follows from Lemma \ref{Lvscdf}. \end{proof}

\begin{lemma}
Let
\be
H^n(t_{k}): =\delta^3+\omega(\delta, 0) \delta^2
+\delta h^n(t_{k})+\delta^2\sum_{i <  k}h^n(t_i),
\ee
 and let $ \Pi$ be defined as in (\ref{hmainrec}). Then there are positive constants  $C_{16},\tilde C_{16},C_{17}$ such that  \be
\mathbb P(  \Pi^n (t_k) >C_{17}H^n(t_{k-1})|\mathcal D(t_k))   \le
\frac{ C_{16}}{\delta}
\exp(-\tilde C_{16}\delta ^5n).\label{tildepineq}
\ee
\end{lemma}
\begin{proof}

 We begin with breaking up $\Pi^n$ as 
\begin{align}
 \Pi^n (t_k) &\le2\max_{l \le  M/\delta}\left|\frac{\tilde \mu_2(t_{k-1},I_l)}{\tilde
N_2(t_{k-1})}\Delta\tilde
 L(t_{k})- \mu^n_1(t_{k-1}, [0,\delta])P^2_k(t_{k-1},I)\right| \\
 &+\sum_{j = 1}^2\max_{l \le  M/\delta}\left|\pi_j^n(t_k,I_l)- \mu^n_1(t_{k-1}, [0,\delta])
P^2_k(t_{k-1},I_l)\right| \nn\\
&:=G^n(t_k)+ \tilde  \Pi^n(t_k). \nn
\end{align}

From Lemma \ref{pandlbounds},  we can rewrite $\mu_1^n(t_{k-1},[0,\delta])$ in terms of $\underline L^n(t_{k-1})$
and  $\bar L^n(t_{k-1})$, and also   $P^2_k(t_{k-1},I_l)$ in terms of  $\underline
p_l(t_{k-1})$ and $\bar p_l(t_{k-1})$,  from which we can then bound $\tilde \Pi^n(t_k)$, for some $C_{14}>0$, by

\begin{align}
&\mathbb P(\tilde \Pi^n(t_k)>\varepsilon|\mathcal D(t_k)) \label{justpi}\\
&\le \sum_{j = 1}^2 \mathbb P(\max_{l}\left|\pi_j^n(t_k,I_l)- \mu^n_1(t_{k-1},
[0,\delta])
P^2_k(t_{k-1},I_l)\right|>\varepsilon/2|\mathcal D(t_k)) \nn\\
&\le \sum_{j = 1}^2 \Big[ \mathbb P(\max_{l }(\pi_j^n(t_k,I_l)-
\mu^n_1(t_{k-1},
[0,\delta])
P^2_k(t_{k-1},I_l))>\varepsilon/2|\mathcal D(t_k))\nn\\
&+ \mathbb P(\min_{l }(\pi_j^n(t_k,I_l)-
\mu^n_1(t_{k-1},
[0,\delta])
P^2_k(t_{k-1},I_l))<-\varepsilon/2|\mathcal D(t_k))\Big]\nn \\
& \le \sum_{j = 1}^2 \Big[ \mathbb P(\max_{l }(\pi_j^n(t_k,I_l)-
\bar L^n(t_k)\bar
p^n_l(t_{k-1}))+C_{14}H(t_{k-1})>\varepsilon/2|\mathcal D(t_k))\nn\\
&+ \mathbb P(\min_{l }(\pi_j^n(t_k,I_l)-\underline L^n(t_k)\underline
p^n_l(t_{k-1}))-C_{14}H(t_{k-1})<-\varepsilon/2|\mathcal D(t_k))\Big]\nn. \end{align}

 For all paths in $\mathcal D(t_{k}),$ by calculations similar to (\ref{tri1})-(\ref{hrecur}), there exists  a positive constant $C_{15}>C_{14}$   such that
  \begin{align}
& G^n(t_k) \le C_{15}H^n(t_{k-1}).   
\end{align}
and thus
\begin{align}
\mathbb P( \Pi^n (t_k) >8C_{15}H^n(t_{k-1})|\mathcal D(t_{k}))
\le \mathbb P( \tilde \Pi^n(t_k)>4C_{15}H^n(t_{k-1})|\mathcal D(t_{k})). \label{pitilde}
\end{align}

Consider a path $\omega:[0,t_{k-1}]\rightarrow E$ such that $\omega \in \mathcal A(t_{k-1})^c$.  We now invoke (\ref{pitilde}), (\ref{justpi}), and  Lemma \ref{lohidevroye2} with  $\varepsilon(t_{k-1}) = C_{15}H^n(t_{k-1})$,
$C_{17} = 8C_{15}$, and $r(\delta, n) = \frac{C_{12}}{\delta} \exp(-C_{11}\delta ^3n)$ from (\ref{newrbound}) to obtain 
  \begin{align}
& \mathbb P( \Pi^n (t_k)>C_{17}H^n(t_{k-1};\omega)|\mathcal D(t_k)) \\
&\le  \sum_{j = 1}^2\mathbb P\left(\max _{l \le  M/\delta}\left( \pi^n_j(t_k, I_l)-
\bar L^n(t_{k-1};\omega)\bar
p^n_l(t_{k-1};\omega)
\right)> \varepsilon^n(t_{k-1};\mathbf{x})|\mathcal D(t_k) \right)\nn \\
&+ \sum_{j = 1}^2  \mathbb P\left(\min _{l \le  M/\delta}\left( \pi^n_j(t_k, I_l)-
\underline L^n(t_{k-1};\omega)\underline
p^n_l(t_{k-1};\omega)
\right)<- \varepsilon^n(t_{k-1};\mathbf{x})|\mathcal D(t_k) \right)\nn \\
&\le \frac{8M}{\delta}
\exp(-2C_{15}^2 H^n(t_{k-1};\omega)^2n/\bar L^n(t_{k-1};\omega))\\
&+\frac{2M}{\delta}\left(\frac{r}{1-r}+\frac{2rn\bar
L^n(t_{k-1};\omega)}{1-r}+2rn \bar
L^n(t_{k-1};\omega)+r\right)\\
&:= J_1(t_{k-1},n;\omega) +J_2(t_{k-1},n;\omega).
\end{align}
By increasing 
$n^p$ used in obtain (\ref{mbarbound}), if necessary, elementary calculations show that  for  $n>n^{p}$, 
\begin{equation}
J_2(t_{k-1},n;\omega)\le \frac{18M}{\delta^2} \exp(-C_{11}\delta
^3n/2) \cdot \bar
L^n(t_{k-1};\omega). 
\end{equation}
On the other hand, since $\mathcal D(t_k) \subset \mathcal A(t_{k-1})^c$, for $X(t_{k-1}) = \omega' \in \mathcal A(t_{k-1})$ we have \

\begin{equation}
\mathbb P( \Pi^n (t_k)>C_{17}H^n(t_{k-1};\omega')|\mathcal D(t_k)) = 0.
\end{equation}
From the law of total probability,
\begin{align}
&\mathbb P(  \Pi^n (t_k)>C_{17}H^n(t_{k-1})|\mathcal D(t_k))\\
&\le \mathbb E[J_1(t_{k-1},n) |\mathcal A(t_{k-1})^c]+\mathbb E[J_2(t_{k-1},n)|\mathcal A(t_{k-1})^c]  \nn\\
&\le \frac{ C_{16}}{\delta}\exp(-\tilde C_{16}\delta ^5n)\nn
\end{align}
for a sufficiently
small $\tilde C_{16}>0$ and sufficiently large $C_{16}>0$. In the last inequality, we used the simple pathwise bound of $H^n(t_k) \ge \delta^3$ and that  $\bar L = \mathcal O(\delta) $ under $\mathcal A(t_{k-1})^c$. 
\end{proof}

\subsubsection{Proof of Theorem \ref{finalh}}

We consider the events
\begin{align}
\mathcal H(t_k) = \cup_{l\le k} \{h^n(t_k) >  h^n(t_{k-1})+C_{17} H^n(t_{k-1})\},\\
\mathcal B(t_k;C) =\{ d(\tilde \mu(t_k),  \mu^n(t_k)) > C(\delta +\omega(\delta, 0))\}. 
\end{align}
Using the same argument given in Lemma \ref{hasympt},  under $\cap_{t_k\le T_2} \mathcal H(t_k)^c$ and a sufficiently large $C_{18}$,

\begin{equation}
 h^n(t_k) \le C_{18}  \mathcal ( \delta ^2+ \delta \omega(\delta,0)) \quad t_k \le T_2. \label{hnbound}
 \end{equation}We may then use (\ref{hnbound}) with Lemma \ref{ksgrid} to show that
for  a sufficiently large $C_{19}$,
\begin{equation}
\cap_{t_k\le T_2}\mathcal H(t_k)^c \subseteq \cap_{l\le k}
\mathcal B(t_k;C_{19})^c
\end{equation}
and that for sufficiently large $C_{20}$ and small $\tilde C_{20}$, 
\begin{align}
&\mathbb P\left(
\max_{t_k\le T_2}d(\tilde \mu(t_k),  \mu^n(t_k))  > C_{19} ( \delta +  \omega(\delta))\right) =  \mathbb P(\cup_{t_k \le T_2}
 \mathcal B(t_k;C_{19})) \label{almost}\\&\le \sum_{t_k \le T_2} \mathbb P(\mathcal H(t_k)|\mathcal D(t_k))
+\mathbb P(\mathcal D(t_k)^{c}) \nn\\
&\le \sum_{t_k \le T_2}  \mathbb P(  \Pi^n (t_k) >C_{17}H^n(t_{k-1})|\mathcal D(t_k))
+\mathbb
P(\mathcal D(t_k)^{c})\nn \\
&\le\frac{ C_{20}}{\delta^2}
\exp(-\tilde C_{20}\delta ^5n). \nn
\end{align}

To conclude, we may replace the maximum in (\ref{almost}) with a supremum. Indeed,
since $\tilde \mu(t)$ is constant during time intervals $\Delta t_k$, for $t \le T_2$ and   $K = \lfloor \frac{t}{\delta}\rfloor$, 
\begin{equation}
d(\tilde \mu(t),  \mu^n(t)) \le d(\tilde \mu(t_{K}),  \mu^n(t_{K}))+d( \mu^{n}(t),  \mu^n(t_{K})). 
\end{equation}
During $\Delta t_k$, an interval can change its total number by at most $\sum_{j = 1,2} \pi_j(t_k,I)$, and thus 

\begin{equation}
d(\mu^{n}(t),  \mu^n(t_{K})) \le \sum_{\substack{l \le M/\delta\\j = 1,2}} \pi_j(t_k, I_{l}) \le \frac{ M}{\delta} \sup_{\substack{l \le M/\delta\\j
= 1,2}} \pi_j(t_k, I_{l}).
\end{equation}
Then for sufficiently large $C_p> 2C_{19}$ and $C_{2}^p$, and sufficiently small $C_3^p$.
\begin{align}
&\mathbb P\left(
\sup_{t\le T_{2}}d(\tilde \mu(t),  \mu^n(t))  > C^p ( \delta +  \omega(\delta))\right)\label{probdistres}\\
&\le \mathbb P\left(\max_{t_{k}\le T_{2}}d(\tilde \mu(t_{k}),  \mu^n(t_{k}))>  C _{19}( \delta +  \omega(\delta))\right)+\sum_{j = 1}^2 \mathbb P\left(\sup_{\substack{l
\le M/\delta\\t_k \le  T_2}} \pi_j(t_k, I_{l})>   \frac{C _{p}\delta^{2}}{2M}\right)\nn\\
&\le \frac{ C_{2}^p}{\delta^2}
\exp(- C_{3}^p\delta ^5n). \nn
\end{align}

Finally, through a similar stitching argument argument appealed to in the proof of Theorem \ref{detscheme} at the end of Section \ref{disc}, we may  extend (\ref{probdistres}) to any $T'<T(\bar f)$, which yields Theorem  \ref{finalh}.

 \section{Proof of Theorem \ref{solution}} \label{prooftheo1}

For deriving an explicit solution for (\ref{kinint1})-(\ref{kinint2}), we
assume $(f_1(x,t), f_2(x,t))\in Z^2$ for $t \in T(\bar f)$, and show that
such a solution must be given by the explicit expressions (\ref{aexp})-(\ref{f2exp}),
and later verify  that this solution is in fact in $Z^2$. 

We begin by integrating (\ref{kinint2}) over space, giving
\begin{equation}
N_2(t) = N_2(0)-L(t). \label{formn2}
\end{equation}
This implies that $N_2(t)$ is differentiable, with 
\begin{equation}
\dot N_2(t) = -a(t). \label{ndotis}
\end{equation} 
Substituting (\ref{ndotis}) into  (\ref{kinint2}) yields the simple form
\begin{equation}
f_2(x,t) = \frac{N_2(t)}{N_2(0)}\bar f_2(x). \label{f2expr}
\end{equation}
Another substitution of (\ref{f2expr}) into (\ref{kinint1}) then gives
\begin{align}
f_1(x,t) = \bar f_1(x+t)+\int_0^t \frac{\bar
f_2(x+t-s)}{N_2(0)} a(s)ds. 
\end{align}

It remains to express $a(t)$ and $N_2(t)$ in terms of initial conditions.
Setting $x = 0$, we arrive at the closed equation
\begin{equation}\label{fluxeqn}
a(t) =\bar f_1(t)+\int_0^t \frac{\bar f_2(x+t-s)}{N_2(0)} a(s)ds.
\end{equation}
Denoting the probability density $\hat f_2(s) = \frac{\bar f_2(s)}{N_2(0)}$,
we may rewrite (\ref{fluxeqn}) as the integral equation
\begin{equation}\label{renewaleqn}
a(t) = \bar f_1(t)+\int_0^t  a(t-s)\hat f_2(s)ds.
\end{equation}
Equation (\ref{renewaleqn}) is a renewal equation, which has been studied
extensively in probability theory (see \cite[Chapter XI]{feller1974introduction}
for an introduction).  
It is well-known that 
there exists a unique solution for (\ref{renewaleqn}) given by
\begin{align}
a(t) &= \sum_{j = 0}^\infty \hat f_2^{*(j)}(t)*\bar f_1(t)
:= Q_{\hat f_2}(t)*\bar f_1(t),
\end{align}
where the exponent   $*(k)$ denotes $k$-fold self-convolution.  Then (\ref{nform})
follows directly  from (\ref{formn2}). For a locally
bounded density $p$, the  operator $Q_p(t)$   is also locally bounded, (see
Thm. 3.18
 of~\cite{liao2013applied}) .  Thus, it is clear that $a(t)$ and $N(t)$ are
both positive, locally bounded, and continuous for $0\le t < T(\bar f)$,
and subsequently that $(f_1(x,t), f_2(x,t))\in Z^2$. This completes the derivation
for  part (a) for Theorem \ref{solution}.

For showing part (b), the only ambiguity is in establishing for a fixed $t
\in [0,T(\bar f))$, the map $p \mapsto Q_p(t)$ is in $C(L^1([0,T(\bar f)),
\mathbb R_+)$. We  will use a probabilistic argument. Let $X_i^{(p)}$ $i
= 1,2,\dots$ denote a sequence of $[0,\infty)$ valued, iid random variables,
each
with a probability density $p \in L^1(\mathbb R_+)$.  The number of renewals
up to time $t <\infty$ is given by
\begin{equation}
N_p(t) = \sup\left\{k:\sum_{i = 1}^k X_i^{(p)} \le t\right\}.
\end{equation}
 In renewal theory, $Q_p(t)$ is the well-known  \textit{renewal density},
satisfying
\begin{equation}
\int_0^tQ_p(t) = \mathbb E[N_p(t)]-1.
\end{equation}
Each term in the sum of $Q_p(t)$ also has a probabilistic interpretation,
with 
\begin{equation}\label{probsumless}
c_k^{(p)}(t):=\int_0^t f^{*(k)}(t) = \mathbb P\left(\sum_{i = 1}^k X_i^{(p)}
\le t\right), \quad k \ge 1.
\end{equation}
Estimates for the decay of $c_k^{(p)}(t)$ as $k \rightarrow \infty$
can be obtained from Markov's inequality, with 
\begin{align}
 c_k^{(p)}(t) &=  \mathbb P\left(\exp\left(-\sum_{i
= 1}^k X_i^{(p)}\right) \ge e^{-t}\right) \label{sumjust1}\\
&\le e^t \mathbb E[\exp(-X_1^{(p)})]^k. \label{sumjust2}
\end{align}
As  $X^{(p)}$ is a non-deficient random variable,
 $\mathbb P(X = 0) \neq 1.$ Then    $\mathbb E[\exp(-X_1^{(p)})]<1$,
and thus $c_k^{(p)}(t)$    decays exponentially  as $k\rightarrow \infty$.

 To show continuity, fix $p \in L^1(\mathbb R_+)$, $t \in [0, T(\bar f))$,
and   let $\varepsilon >0$. From (\ref{sumjust1})-(\ref{sumjust2}), $c_k^{p}(t)$
is summable in $k$, so we may choose a $K>0$ such that 

\begin{equation}
\sum_{i = K}^\infty c_k^{(p)}(t)< \varepsilon/6.
\end{equation}
Since $\mathbb E[\exp(-X_1^{(p)})]$ varies
continuously with respect with
$p$ in $L^1(\mathbb R_+)$,  a similar calculation to (\ref{sumjust1})-(\ref{sumjust2})
implies that the map  $p \mapsto c_k^{(p)}(t)$ is also in $ C( L^1(\mathbb
R_+), \mathbb R_+)$ for all $k \ge 1$. Furthermore, tail sums of $c_k^{(p)}(t)$
also vary continuously in the $p$ variable.    Thus, there exists   $\delta>0$
such that if $\tilde p\in L^1(\mathbb R_+)$ satisfies $\|p-\tilde
p\| _{L^1(\mathbb R_+)} < \delta$, then both  
\begin{equation}
\sum_{i = K}^\infty c_k^{(\tilde p)}(t) < \varepsilon/3 \quad \hbox{and}
\quad  \sum_{i = 1}^{K-1} |c^{(p)}_k(t)-  c^{(\tilde p)}_k(t)| < \varepsilon/2
\end{equation}

hold. It then follows that
\begin{align}
 |Q_p(t)-Q_{\tilde p}(t)| &\le  \sum_{i = 1}^{K-1} |c^{(p)}_k(t)-  c^{(\tilde
p)}_k(t)|+\sum_{i = K}^\infty c_k^{( p)}(t)+\sum_{i = K}^\infty c_k^{(\tilde
p)}(t)<\varepsilon.
 \end{align}

\textbf{Acknowledgments:}{ The author wishes to thank Anthony Kearsley and Paul Patrone for providing guidance during his time as a National Research Council Postdoctoral Fellow at the National Institute of Standards and Technology, and also Govind Menon for several discussions regarding the manuscript.}

\bibliographystyle{siam}
\bibliography{refs}
\end{document}